\documentclass[reqno]{amsart}

\makeatletter
\renewcommand\part{%
	\if@noskipsec \leavevmode \fi
	\par
	\addvspace{4ex}%
	\@afterindentfalse
	\secdef\@part\@spart}

\def\@part[#1]#2{%
	\ifnum \c@secnumdepth >\m@ne
	\refstepcounter{part}%
	\addcontentsline{toc}{part}{\thepart\hspace{1em}#1}%
	\else
	\addcontentsline{toc}{part}{#1}%
	\fi
	{\parindent \z@ \raggedright
		\interlinepenalty \@M
		\normalfontsymmetric decreasing rearrangement
		\ifnum \c@secnumdepth >\m@ne
		\Large\bfseries \partname\nobreakspace\thepart
		\par\nobreak
		\fi
		\huge \bfseries #2%
		\par}%
	\nobreak
	\vskip 3ex
	\@afterheading}
\def\@spart#1{%
	{\parindent \z@ \raggedright
		\interlinepenalty \@M
		\normalfont
		\huge \bfseries #1\par}%
	\nobreak
	\vskip 3ex
	\@afterheading}
\makeatother
\usepackage{color}
\usepackage[dvipsnames]{xcolor}
\usepackage{ifpdf}
\ifpdf 
    \usepackage[pdftex]{graphicx}   
    \usepackage[pdftex,     
            plainpages=false,   
            breaklinks=true,    
            colorlinks=true,
            linkcolor=red,
            citecolor=green,
            pdftitle=My Document
            pdfauthor=My Good Self
           ]{hyperref} 
\else 
    \usepackage{graphicx}       
    \usepackage{hyperref}       
\fi 

\usepackage{subfig}
\usepackage{glossaries}
\usepackage{glossary-mcols}

\usepackage{aurical}
\usepackage{amsfonts,amsmath}	
\usepackage{amssymb}
\usepackage{verbatim}
\usepackage{amsopn}
\usepackage[english]{babel}
\usepackage{amsthm}
\usepackage{enumerate}
\usepackage{mathrsfs}	
\usepackage{enumitem}
\usepackage{mathtools}
\usepackage{esint}
\usepackage{bbm}
\usepackage{caption}
\usepackage{marginnote}
\usepackage[marginparwidth=2cm]{geometry}
\usepackage{csquotes}

\date{\today}

\theoremstyle{definition} \newtheorem{definition}{Definition}[section]
\theoremstyle{definition} \newtheorem{remark}[definition]{Remark}
\theoremstyle{plain} \newtheorem{lemma}[definition]{Lemma}
\theoremstyle{plain} \newtheorem{proposition}[definition]{Proposition}
\theoremstyle{plain} \newtheorem{theorem}[definition]{Theorem}
\theoremstyle{plain} \newtheorem{corollary}[definition]{Corollary}
\theoremstyle{definition} 
\theoremstyle{plain} 
\theoremstyle{definition} 
\theoremstyle{definition}

\newcommand{\R}{\mathbb{R}}

\newcommand{\loc}{\text{\rm loc}}

\newcommand{\ind}{1\!\!\mathrm{I}}

\newcommand{\hb}{\tilde b}

\numberwithin{equation}{section} 

\theoremstyle{plain} \newtheorem*{theorem*}{Theorem}
\theoremstyle{plain} 
\theoremstyle{plain} \newtheorem*{mthm*}{Main Theorem}
\theoremstyle{plain} \newtheorem*{conjecture*}{Conjecture}
\theoremstyle{plain} 
\theoremstyle{plain} \newtheorem*{problem*}{Problem}

\usepackage{biblatex} 
\title{Existence and blow-up for non-autonomous scalar conservation laws with viscosity}

\author{Stefano Bianchini}
\address{S. Bianchini: S.I.S.S.A., via Bonomea 265, 34136 Trieste, Italy}
\email{bianchin@sissa.it}

\author{Giacomo Maria Leccese}
\address{G.M. Leccese: S.I.S.S.A., via Bonomea 265, 34136 Trieste, Italy}
\email{gleccese@sissa.it}

\begin{document}


\begin{abstract}
We consider a question posed in \cite{guidolin2022global}, namely the blow-up of the PDE
\begin{equation*}
    u_t + (b(t,x) u^{1+k})_x = u_{xx}
\end{equation*}
when $b$ is uniformly bounded, Lipschitz and $k = 2$. We give a complete answer to the behavior of solutions when $b$ belongs to the Lorentz spaces $b \in L^{p,\infty}$, $p \in (2,\infty]$, or $b_x \in L^{p,\infty}$, $p \in (1,\infty]$.
\end{abstract}

%
%
\maketitle



%


\section{Introduction}

In this paper, we study the global in time existence and long time behavior for the initial value problem
\begin{equation}\label{m:eq:pde}
\begin{cases}
    u_t + (b(t,x) u^{k+1})_x = u_{xx},\quad x\in\R,\,t \in (0,\infty),\\
    u(0,x)=u_0(x),
\end{cases}
\end{equation}
where $b(t,x)$ is a non-autonomous drift and $u_0$ is the initial datum $u_0$. This question was raised in \cite{guidolin2022global}, where the subcritical case was analyzed. It is not restrictive to assume that
\begin{equation*}
    u(t,x) \geq 0,
\end{equation*}
because the PDE \eqref{m:eq:pde} is monotone w.r.t. the initial data. Moreover by scaling $a u(a^2 t,ax)$ we can assume that $\|u_0\|_1 = \|u(t)\|_1 = 1$, where we have used that the PDE is in divergence form.

In this work we consider drifts $b$ which are in weak-$L^P$ or with derivative in weak-$L^p$. More precisely, we make the following assumptions on the initial data $u_0$ and exponent $k$: 
\begin{enumerate}
\item $b$ is only integrable:
\begin{equation}\label{eq:con1}
b\in L^\infty_{\text{loc}} ((0,\infty), L^{p,\infty}(\R)),\ p\in(2,\infty], \quad \begin{cases}
k> 0,\\
u_0\in L^1(\R)\cap L^\infty(\R);
\end{cases}
\end{equation}
\item $b$ has a weak derivative in $x$:
\begin{equation}\label{eq:con2}
\begin{cases}
    b \in L^\infty_\text{loc}((0,\infty) \times \R), \\
b_x \in L^\infty_{\text{loc}}((0,\infty), L^{p,\infty}(\R)),\ p\in(1,\infty], 
\end{cases}
\begin{cases}
k \geq \frac{1}{2},\\
u_0\in  L^1(\R)\cap L^\infty(\R),\\ 
E(u_0):=\int |x|^2|u_0(x)|\,dx <+\infty.
\end{cases}
\end{equation}
\end{enumerate}
The space $L^{p,\infty}$ is the standard Lorentz space, see Subsection \ref{sec:lorentz} for the precise definition: we just recall here that $L^{p,\infty}$ is also referred to as the weak-$L^p$ space.

The case $k=0$ corresponds to a linear PDE, which can be studied by means of the Duhamel formula: hence in this paper we restrict ourselves to $k > 0$. The assumption $k \geq \frac{1}{2}$ in \eqref{eq:con2} is due to the fact that the drift $b$ may be unbounded. If $b$ is uniformly bounded that one can remove this assumption. It has actually no influence in the blow-up behaviour, which is a local property.

The aim is to investigate the relation between the exponents $k$ and $p$ so that the solution exists globally in $L^\infty$ and to study the behaviour of $u(t)$ behavior for $t \to \infty$. The results of this paper about the existence of a bounded solution can be summarized in the following theorem.

\begin{theorem}
\label{Theo:existence_1}
    Assume that the drift $b$ and the initial condition $u_0$ satisfy \eqref{eq:con1} with $k\le 1-\frac{1}{p}$ or satisfy \eqref{eq:con2} with $k \le 2-\frac{1}{p}$. Then solution $u(t)$ of \eqref{m:eq:pde} is globally defined $[0,\infty)$.
    
    Conversely, assume that $b,u_0$ satisfy \eqref{eq:con1} with $k> 1 - \frac{1}{p}$ or satisfy \eqref{eq:con2} with $k> 2-\frac{1}{p}$. Then there are bounded initial data such that the corresponding solutions of \eqref{m:eq:pde} blow up in finite time.
\end{theorem}

The first part of the above statement is contained in Theorem \ref{Theo:not_blow_up}, Section \ref{Ss:bounds_L_infty}. The second part instead in given in Theorems \ref{prop:blowup} and \ref{prop:blowup:2}, Section \ref{s:blowup}.

The analysis of the subcritical case $k < 2$ and $b$ bounded, Lipschitz has been done in \cite{guidolin2022global}. The above theorem extends their results to other class of drift $b$. The blow-up results follow the analysis in \cite{toscani2012finite}, where a specific drift is considered in the multidimensional case is considered with $k = 2$.

Concerning the long time behavior, we assume uniform bounds on $b$, i.e. we strengthen the conditions on $b$ as follows:
\begin{enumerate}
\item under the assumptions \eqref{eq:con1}, we also require
\begin{equation}
\label{Equa:con1}
    b \in L^\infty((0,\infty), L^{p,\infty}(\R)), \quad p\in(2,\infty];
\end{equation}
\item under the assumptions \eqref{eq:con2}, we also require
\begin{equation}
    \label{Equa:con2}
    b \in L^\infty((0,\infty) \times \R), \ b_x\in L^\infty((0,\infty), L^{p,\infty}(\R)), \quad p\in(1,\infty].
\end{equation}
\end{enumerate}

We obtained the following results (Theorem \ref{Theo:long_time_1}, Section \ref{s:long:behav}):

\begin{theorem}
\label{Theop:longtime_1}
Assume that the solution is bounded for all $t > 0$ and the conditions \eqref{eq:con1}, \eqref{Equa:con1}, $k \geq 1-\frac{1}{p}$ hold, or \eqref{eq:con2}, \eqref{Equa:con2}, $k \geq 2-\frac{1}{p}$ hold. Then $u(t)$ decays to $0$ uniformly as $t \to \infty$ as $t^{-\frac{1}{2}}$.

    Viceversa, assume \eqref{eq:con1}, \eqref{Equa:con1} with $k < 1 - \frac{1}{p}$ or \eqref{eq:con2}, \eqref{Equa:con1} with $k < 2 - \frac{1}{p}$. Then there are drifts $b$ and initial data $u_0$ such that the corresponding solutions of \eqref{m:eq:pde} do not decay to $0$ as $t \to \infty$.
\end{theorem}

In particular, in the case $b_x \in L^\infty_{\text{loc}}((0,\infty), L^{\infty}(\R))$ we answer to a question raised in \cite{guidolin2022global}, precisely Question 3:
\begin{equation}
\label{long2}
\begin{split}
\text{``Is it possible to guarantee global existence for solutions of the problem (1) when $k\ge2$, $p = \infty$?"}
\end{split}
\end{equation}
The problem (1) referred above in \eqref{long2} and considered there is the PDE
\begin{equation}
\label{Equa:PDE_guidolin}
    \begin{cases}
        u_t + (b(x,t)u^{k+1})_x = \mu(t) u_{xx},\\
        u_0\in L^1(\R)\cap L^\infty,\ u_0\ge 0
    \end{cases}
\end{equation}
with $\mu$ strictly positive continuous function and $b$ is uniformly bounded and Lipschitz. The PDE \eqref{Equa:PDE_guidolin} and \eqref{m:eq:pde} are equivalent because of the following time transformation:
\begin{equation*}
u(t,x) = v(\tau(t),x), \quad \dot \tau(t) = \mu(t), 
\quad \tau(0) = 0,
\end{equation*}
which leads to the equation \eqref{m:eq:pde}, namely
\begin{equation*}
    v_t+(\tilde b(\tau,x)v^{k+1})_x=v_{xx},\quad \tilde b(\tau(t),x)=\frac{b(t,x)}{\mu(t)}.
\end{equation*}

The results of the above theorem can be summarized in the following table: setting
\begin{equation*}
     \mathbf{critic} (p)=
     \begin{cases}
         1-1/p  &b\in L^{p,\infty},\\
         2-1/p &b_x\in L^{p,\infty},
     \end{cases} 
 \end{equation*}
we have
\begin{table}[h]
\begin{tabular}
{l|c|c|c}
 & $k < \mathbf{critic} (p)$ & $k = \mathbf{critic} (p)$ & $k >\mathbf{critic} (p)$ \\
\hline
$\begin{matrix}b \in L^{p,\infty},\ {2<p\le\infty} \\
\text{or } b_x \in L^{p,\infty}, 1<p\le \infty
\end{matrix}$& $\begin{array}{c}
\text{global existence} \\ \text{no decay in general}
\end{array}$ & $\begin{array}{c} \text{global existence} \\ \text{and decay as $t^{-\frac{1}{2}}$} \end{array}$ & $\begin{array}{c}
\text{blow-up in finite time,} \\ \text{if bounded then decay as $t^{-\frac{1}{2}}$}
\end{array}$
\end{tabular}
\end{table}

\subsection{Structure of the paper}

This article is organized as follows.

In Section \ref{S:preliminaries} we introduce some definitions, notations, and well known results on Lorentz spaces (Section~\ref{sec:lorentz}): comparisons, embeddings, interpolations, H\"older's inequalities and convolutions estimates. Since we are using multiplication/convolutions operators and embedding, the Lorentz space setting more or less gives the same estimates as for $L^p$. We also recall a special case of Gagliardo-Niremberg inequality and prove a useful estimate on the heat kernel.

In Section \ref{s:loc:exist} we recall the local existence and uniqueness of the solutions via Duhamel's principle. The assumption $k \geq \frac{1}{2}$ and that $E(u_0) < \infty$ enters only in this section, and are needed if we let $b$ to be unbounded. The results in this section are standard, and independent on the main theorems of the paper.

The main idea of the paper is contained in Section \ref{s:glob}, where we deal with the global existence of the solutions. Differently from the approach of \cite{guidolin2022global}, we use a standard rescaling the solutions about the blow-up point at time $T$, Section \ref{Ss:realed_variables}. By means of energy estimates for the truncated solution (Lemma \ref{lemma:entropty}) and Gagliardo-Niremberg inequality (Lemma \ref{lemma:entropty2}), we show that in the new variables $(\tau,y) \in \R^+ \times \R$ that the rescaled solution decays the $L^2$-norm as the Heat kernel $\tau^{-\frac{1}{4}}$ (Lemma \ref{Lem:decay_L_2_case_1}). This fact will lead to a uniform estimate on the $L^2$-norm of the original solution $u$ (Corollary \ref{Cor:bound_u_2}), and by adapting the estimate \cite[Theorem 3.8]{guidolin2022global} to our case we deduce that $\|u\|_\infty$ is bounded, Theorem \ref{Theo:not_blow_up} of Section \ref{Ss:bounds_L_infty}. This concludes the proof of the first part of Theorem \ref{Theo:existence_1}. \\
In Section \ref{Ss:critical_case_integr} we study the critical cases, and show that the fact that the norm of $b$ remains constant under rescaling leads to a uniform decay rate, Theorem \ref{Theo:decay_bx_critical}.

In Section \ref{s:blowup} we provide examples of solutions blowing up in finite time for $k$ above the critical value. The ideas are taken from \cite{toscani2012finite} and adapted to our situation. Theorem \ref{prop:blowup} corresponds to the second part of Theorem \ref{Theo:existence_1} for $k > 1 - \frac{1}{p}$ and $b(t) \in L^{p,\infty}$, while the other case $k > 2 - \frac{1}{p}$ and $b_x(t) \in L^{p,\infty}$ is in the statement of Theorem \ref{prop:blowup:2}. As observed already in \cite{toscani2012finite}, we notice in Remark \ref{Rem:blow_up_always} that the $L^1$-norm of the initial data cannot be too small, otherwise blow-up is not possible.

In Section \ref{s:long:behav} we discuss the long behavior of solutions, proving Theorem \ref{Theop:longtime_1}. The proof of the main result of this section, Theorem \ref{Theo:long_time_1}, gives examples of bounded solutions in the subcritical case, and by adapting the analysis of the decay for the critical case we obtain that the solution decay in the critical or supercritical case, if we assume that $u \in L^\infty_{t,x}$.

\medskip

\noindent {\bf Acknowledgment.} This research has been partially supported by the project PRIN 2020 "Nonlinear evolution PDEs, fluid dynamics and transport equations: theoretical foundations and applications".

\section{Preliminaries}
\label{S:preliminaries}

Given a function $f \in L^1_\loc(\R)$, and $\alpha \geq 0$ we define the $\alpha$-moment of $f$ as
    \[
    M_\alpha(f)=\int |x|^\alpha |f| \,dx.
    \] 
In particular we will use the notation 
$$
m(f) = M_0(f), \ \text{as the \textit{mass} of $f$,} \qquad \ E(f)=M_2(f), \ \text{as the \textit{energy} of $f$.}
$$

We will write
$$(f\wedge g)(x):=\min \{f(x),g(x)\},\qquad (f\vee g)(x):=\max \{f(x),g(x)\}.$$

The letter $C$ will be a constant that could change line by line, also we use the symbol $a\lesssim b$ as shorthand for $a\le C b$ for some constant $C$. We will write $a\simeq b$ if both $a \lesssim b$ and  $b\lesssim a$ hold.

The symbol $\star$ will denote the convolution operation:
\begin{equation*}
    (f \star g)(x) = \int f(y) g(x-y) dy.
\end{equation*}

\subsection{Lorentz spaces}
\label{sec:lorentz}

We briefly recall the definition and some results about Lorentz space, see \cite{o1963convolution}.

Let $f^*$ be the symmetric decreasing rearrangement defined by
\begin{equation*}
    f^{\ast}(x) = \inf \Big\{\alpha>0: \mathcal L^1\big(\big\{|f|>\alpha\big\}\big)|\leq x\Big\},
\end{equation*}
and let $f^{**}(x)$ be the function defined by
\begin{equation*}
    f^{**}(x) = \frac{1}{x} \int_0^x f^*(y) dy.
\end{equation*}
Define 
\begin{equation}
\label{Equa:Lorentz_1}
\|f\|_{p,q} = \begin{cases} 
\left( \displaystyle \int_0^{\infty} \big[ x^{\frac{1}{p}} f^{\ast \ast}(x) \big]^q \, \frac{dx}{x} \right)^{\frac{1}{q}} & q \in [1, \infty), p \in (1,\infty) \\
\sup\limits_{x > 0} \, x^{\frac{1}{p}}  f^{\ast\ast}(x)   & q = \infty, p \in [1,\infty].
\end{cases}
\end{equation}
Note that
\begin{equation*}
    \|f\|_{1,\infty} = \|f^*\|_1 = \|f\|_1, \quad \|f\|_{\infty,\infty} = \lim_{x \searrow 0} f^{**}(x) = \|f\|_\infty.
\end{equation*}

By Hardy's inequality
\begin{equation*}
    \bigg[ \int_0^\infty \bigg( x^{1/p-1} \int_0^x |f(t)| dt \bigg)^q \frac{dx}{x} \bigg]^{1/q} \leq p' \bigg[ \int_0^\infty \big( x^{1/p} |f(x)| \big)^q \frac{dx}{x} \bigg]^{1/q}, \quad \frac{1}{p} + \frac{1}{p'} = 1, \ p \in (1,\infty),
\end{equation*}
and some fairly easy computations, one can verify that for $1 < p < \infty$ the above definition is equivalent to
\begin{equation}
\label{Equa:Lorentz_2}
|||f|||_{p,q} = \begin{cases}
\left( \displaystyle \int_0^{\infty} \big[ x^{\frac{1}{p}} f^{*}(x) \big]^q \, \frac{dx}{x} \right)^{\frac{1}{q}} & q \in [1, \infty), p \in [1,\infty), \\
\sup\limits_{x > 0} \, x^{\frac{1}{p}}  f^{*}(x)   & q = \infty, p \in [1,\infty].
\end{cases}
\end{equation}
For $p=q = \infty$ the equivalence is elementary. For $p = 1, q = \infty$ the quantity in \eqref{Equa:Lorentz_2} is the weak-$L^1$ norm, while \eqref{Equa:Lorentz_1} corresponds to the $L^1$-space: the $L^1$-norm in the above definition is realized for $p=q=1$. It is elementary to deduce that for $p \in (1,\infty)$
\begin{equation}
\label{Equa:estima_scal}
    \|f^k\|_{p,q} \leq p' |||f^k|||_{p,q} = p' |||f|||_{kp,kq}^k \leq p' \|f\|_{kp,kq}^k. 
\end{equation}
It is immediate to check that
$$
\|f^k\|_{\infty,\infty} = \|f^k\|_\infty = \|f\|_\infty^k = \|f\|_{\infty,\infty}^k.
$$

\begin{definition}[Lorentz space]
The \emph{Lorentz space $L^{p,q}$} is the space of (equivalence classes of) measurable functions $f$ such that $\|f\|_{p,q} < \infty$. It is a Banach space with the norm $\|\cdot\|_{p,q}$.
\end{definition}
 
We will use the following results.

\begin{proposition}[{\cite[Lemma 2.2]{o1963convolution}}]
\label{prop:equiv_L_p}
Let $1 < p<\infty$, then
$$
\|f\|_p \leq \|f\|_{p,p} \leq p' \|f\|_p, \qquad \text{where} \quad \frac{1}{p} + \frac{1}{p'} = 1.
$$
\end{proposition}

\begin{proposition} [{\cite[Lemma 2.5]{o1963convolution}}]\label{prop:inclusion_lorentz}
    Suppose $1 < p < \infty$ and $1\le q<r\le \infty$, then
    \begin{equation*}
        \|f\|_{p,r}\le \left(\frac{q}{p}\right)^{\frac{1}{q} - \frac{1}{r }} \|f\|_{p,q}.
    \end{equation*}  
\end{proposition}

Another elementary estimate we use is the following.

\begin{lemma}
\label{Lem:triv_estim_p_p'}
For $p_1 < p < p_2$ and $q,q_1,q_2 \geq 1$ it holds
\begin{equation*}
\|f\|_{p,q} \leq \bigg( \frac{1}{\frac{q}{p} - \frac{q}{p_2}} + \frac{1}{\frac{q}{p_1} - \frac{q}{p}} \bigg)^{\frac{1}{q}} \bigg( \frac{p}{q_1} \bigg)^{\frac{1}{q_2} \frac{\frac{1}{p} - \frac{1}{p_2}}{\frac{1}{p_1} - \frac{1}{p_2}}} \bigg( \frac{p}{q_2} \bigg)^{\frac{1}{q_1} \frac{\frac{1}{p_1} - \frac{1}{p}}{\frac{1}{p_1} - \frac{1}{p_2}}} \|f\|_{p_1,q_1}^{\frac{\frac{1}{p} - \frac{1}{p_2}}{\frac{1}{p_1} - \frac{1}{p_2}}} \|f\|_{p_2,q_2}^{\frac{\frac{1}{p_1} - \frac{1}{p}}{\frac{1}{p_1} - \frac{1}{p_2}}}.
\end{equation*}
\end{lemma}

\begin{proof}
The fundamental estimate is that, being $f^{**}$ decreasing, then
\begin{equation}
\label{Equa:f**_estima}
    \frac{q}{p} \bar x^{\frac{q}{p}} (f^{**}(\bar x))^q \leq \int_0^{\bar x} \big[ x^{\frac{1}{p}} f^{**}(x) \big]^q \frac{dx}{x} \leq \|f\|_{p,q}^q.
\end{equation}
Hence we can write for $q< \infty$
\begin{equation*}
    \begin{split}
        \|f\|^q_{p,q} &= \bigg[ \int_0^{\bar x} + \int_{\bar x}^\infty \bigg] \big[  x^{\frac{1}{p}} f^{**}(x) \big]^q \frac{dx}{x} \\
\big[ \text{(\ref{Equa:f**_estima}) for $p_1,p_2$} \big] \quad       &\leq \frac{1}{\frac{q}{p} - \frac{q}{p_2}} \bigg( \frac{p}{q_2} \bigg)^{\frac{q}{q_2}} \|f\|^q_{p_2,q_2} \bar x^{\frac{q}{p} - \frac{q}{p_2}} + \frac{1}{\frac{q}{p_1} - \frac{q}{p}} \bigg( \frac{p}{q_1} \bigg)^{\frac{q}{q_1}} \|f\|^q_{p_1,q_1} x^{\frac{q}{p} - \frac{q}{p_1}}.
    \end{split}
\end{equation*}
One can directly check that the same estimate holds for $q = \infty$ as 
\begin{equation*}
    \|f\|_{p,\infty} \leq  \bigg( \frac{p}{q_2} \bigg)^{\frac{1}{q_2}} \|f\|_{p_2,q_2} \bar x^{\frac{1}{p} - \frac{1}{p_2}} + \bigg( \frac{p}{q_1} \bigg)^{\frac{1}{q_1}} \|f\|_{p_1,q_1} x^{\frac{1}{p} - \frac{1}{p_1}},
\end{equation*}
and similarly for $q_1 = \infty$ and/or $q_2 = \infty$.

Optimizing w.r.t. to $\bar x$,
\begin{equation*}
    \bar x^{\frac{1}{p_1} - \frac{1}{p_2}} = \frac{\big(\frac{p}{q_1}\big)^{\frac{1}{q_1}} \|f\|_{p_1,q_1}}{\big(\frac{p}{q_2}\big)^{\frac{1}{q_2}} \|f\|_{p_2,q_2}},
\end{equation*}
one obtains the statement.
\end{proof}

\begin{theorem}[H\"older’s inequality in Lorentz spaces, {\cite[Theorem 3.4, Theorem 3.5]{o1963convolution}}]
\label{Theo:holder_lorentz}
If $1 < p, p_1, p_2 < \infty$ and
$1 \le q,q_1, q_2 \leq \infty$ satisfy
$$
\frac 1 p = \frac 1 {p_1} +\frac 1 {p_2}, \quad  \frac 1 q \le  \frac 1 {q_1} + \frac 1 {q_2},
$$
then
$$
\|fg\|_{p,q} {\leq p'} \|f\|_{p_1,q_1} \|g\|_{p_2,q_2}.
$$
where $1/p + 1/p' = 1$.

If
\begin{equation*}
    1 = \frac{1}{p_1} + \frac{1}{p_2}, \qquad 1 \leq \frac{1}{q_1} + \frac{1}{q_2},
\end{equation*}
then
\begin{equation*}
    \|fg\|_1 \leq \|f\|_{p_1,q_1} \|g\|_{p_2,g_2}.
\end{equation*}
\end{theorem}

\begin{corollary}
    \label{Cor:bound_b_infty}
    If $f_x \in L^{p,\infty}$, $p > 1$, then
    \begin{equation*}
        |f(x) - f(x')| \leq p' \|f\|_{p,\infty} |x - x'|^{ \frac{1}{p'}}.
    \end{equation*}
\end{corollary}

\begin{proof}
For $p= \infty$ the estimate is just the Lipschitz regularity of $f$.

We have by the last formula of Theorem \ref{Theo:holder_lorentz}
\begin{equation*}
    \begin{split}
        |f(x)| &= \bigg| \int f_x(x') \ind_{(0,x)} dx' \bigg| \\
        &\leq \|f_x\|_{p,\infty} \|\ind_{(0,x)}\|_{1 - \frac{1}{p},1} = \|f_x\|_{p,\infty} p'|x|^{\frac{1}{p'}}. \qedhere
    \end{split}
\end{equation*}
\end{proof}

\begin{theorem}[{\cite[Theorem 2.6]{o1963convolution}}]
\label{Theo:young_lorentz}
If $1 <  p, p_1, p_2 < \infty$ and
$1 \le q_1, q_2, q \le \infty$ satisfy
$$
\frac{1}{p} +1 =\frac{1}{p_1} + \frac{1}{p_2}, \quad \frac{1}{q} \le \frac{1}{q_1} + \frac{1}{q_2},
$$
then
$$
\|f\star g\|_{p,q} \leq 3p \|f\|_{p_1,q_1} \|g\|_{p_2,q_2}.
$$
\end{theorem}

We recall also Young's theorem on convolution on $L^p$-spaces:
\begin{equation*}
    \|f \ast g\|_p \leq \|f\|_{p_1} \|g\|_{p_2}, \quad \frac{1}{p} + 1 = \frac{1}{p_1} + \frac{1}{p_2},
\end{equation*}
which holds also in the case $p=1$ (with the constant $3p$ replaced by $1$). The case $p = \infty$ gives also \cite[Theorem 3.6]{o1963convolution}
\begin{equation*}
\begin{split}
    \|f \star g\|_{\infty} &\leq 
    \|f\|_{p,q} \|g\|_{p',q'},
\end{split}
\end{equation*}
with $1/p + 1/p' = 1$, $1/q + 1/q' \geq 1$. We will also use the following variant of the above inequality when $g \in L^\infty$,
\begin{equation*}
    \|f\star g\|_{p,q} \le \|f\|_{p,q} \|g\|_\infty.
\end{equation*}

\subsection{Gagliardo-Nirenberg inequality}

We recall the Gagliardo–Nirenberg interpolation inequality in 1-dimension.

\begin{proposition}[Gagliardo-Nirenberg interpolation inequality,\cite{nirenberg1959elliptic}]
\label{Prop:gagliardo_Lorentz}
Let $1 < q < p$, then
\begin{equation*}
    \|f \|_{p}\lesssim \|f_x\|^\theta_{2}\|f\|_{q}^{1-\theta}
\end{equation*}
 with
 \begin{equation*}
 \frac 1 p = - \frac{\theta}{2}+\frac {1-\theta} q  
 \end{equation*}
\end{proposition}

\subsection{Heat kernel}

Recall that the heat kernel in $\R$ is
$$G(t,x)=\frac {1}{\sqrt{4\pi t}}e^{-{x^{2}/(4t)}}.$$ 
In the following we will use the estimate below. 

\begin{proposition}
\label{prop:appx}
It holds
    $\|G_x\|_{p,1}\le C p' t^{{1}/{2p}-1}$. 
\end{proposition}

\begin{proof}
We estimate 
\begin{equation*}
    |G_x(t)|\le \frac{C}{t}\mathbf 1_{[-\sqrt{2t},\sqrt{2t}]}(x)+\mathbf 1_{\R\setminus [-\sqrt{2t},\sqrt{2t}]}(x) |G_x(t)|, 
\end{equation*}
thus 
\begin{equation*}
\begin{split}
    \|G_x(t)\|_{p,1} &= \int \bigg[ x^{-1-1/p'} \int_0^x (G_x(t,y))^* dy \bigg] dx
    \simeq p' \int G^*_x(t,y) y^{-1/p'} dy \simeq p' t^{\frac{1}{2p} - 1}. 
    \qedhere
\end{split}
\end{equation*}
\end{proof}

\section{Local existence and uniqueness}
\label{s:loc:exist}

For the sake of completeness, in this section we prove some classical results of local existence and uniqueness for $L^\infty \cap L^1$-solutions of \eqref{m:eq:pde}. Define the integral operator $\Phi[u]$ by Duhamel formula
\begin{equation*}
\begin{split}
    u \mapsto \Phi[u](t)&=G(t)\ast u_0 + \int_0^t G_x(t-s) \ast (b(s)u^k(s))\,d s.
\end{split}
\end{equation*}
Notice that $\|u(t)\|_1 = \|u_0\|_1$, as required being the PDE in conservation form.

\begin{proposition}
\label{prop:loc:exs:1}
Let $b \in L^\infty_t L^{p,\infty}_x$, $k \geq 0$, $u_0 \in L^\infty \cap L^1(\R)$ and $\frac{1}{p} + \frac{1}{p'} = 1$. Then $u \mapsto \Phi[u]$ is a contraction in the set
$$
S = \left\{ u \in L^\infty( [0,\bar t] \times \R) \cap C([0,\bar t],L^{p'}(\R)): \|u\|_\infty \le r,\, u(0) = u_0 \right\},
$$
with $\bar t$ sufficiently small and $r \geq 2 \|u_0\|_\infty$. In particular, there is a unique bounded solution for \eqref{m:eq:pde}, and it belongs to the space $L^\infty((0,\bar t) \times \R) \cap C([0,\bar t], L^{p'}(\R))$.
\end{proposition}

	We first prove that $t \mapsto \Phi[u](t)$ is continuous in $L^\infty$ when $u \in S$ and $t > 0$: this will be useful later on, and shows also the continuity of $t \mapsto \Phi[u](t)$ in every integral norm $\|\cdot\|_q$, $q \in [1,\infty]$ for $t > 0$.
	
\begin{lemma}
	\label{Lem:conti_integr}
	For every $u \in L^\infty_{t,x}((0,\bar t) \times \R)$ the function $t \mapsto \Phi[u](t)$ is continuous in $L^\infty(\R)$ for $t > 0$. 
\end{lemma}

\begin{proof}
	Compute
	\begin{equation*}
		\begin{split}
			\|u(t + \delta) - u(t)\|_\infty &= \bigg\| (G(t+\delta) - G(t)) \ast u_0 + \int_0^{t+\delta} G(t+\delta -s) \ast b(s) u^{1+k}(s) ds \\
			&\qquad - \int_0^t G(t - s) \ast b(s) u^{1+k}(s) ds \bigg\|_\infty \\
\big[ \text{Theorem \ref{Theo:holder_lorentz}} \big] \quad		&\leq \|G(t+\delta) - G(t)\|_1 \|u_0\|_\infty + C\int_t^{t+\delta} \|G(t+\delta -s)\|_{p',1} \|b(s) u^{1+k}(s)\|_{p,\infty} ds \\
			& \qquad +C \int_0^t \|G(t + \delta - s) - G(t-s)\|_{p',1} \|b(s) u^{1+k}(s)\|_{p,\infty} ds \\
\big[ \text{Proposition \ref{prop:appx}} \big] \quad			&\leq \|G(t+\delta) - G(t)\|_1 \|u_0\|_\infty \\
& \quad + C \|b\|_{p,\infty} \|u\|_\infty^{k+1} \bigg( \delta^{\frac{1}{2p'}} + \int_0^t \|G(t + \delta - s) - G(t-s)\|_{p',1} ds \bigg).
		\end{split}
	\end{equation*}
		The last terms converge to $0$ as $\delta \to 0$ uniformly in every interval of the form $[t_0,t_1]$, $t_0 > 0$.
\end{proof}

\begin{proof}[Proof of Proposition \ref{prop:loc:exs:1}]
	We start by deducing the uniform continuity in $L^1$ as $t \to 0$: this will give the continuity in time for $\|\cdot\|_q$, $q \in [1,\infty)$. From Theorem \ref{Theo:young_lorentz} in the first inequality we obtain
	\begin{equation*}
\begin{split}
			\|u(t) - u_0\|_1 &\leq \|G(t) \ast u_0 - u_0\|_1 + \int_0^t \|G_x(t-s)\|_1 \|b(s) u^{1+k}(s)\|_1 ds \\
		\big[ \text{Theorem \ref{Theo:holder_lorentz}} \big] \quad &\leq \|G(t) \ast u_0 - u_0\|_1 + \int_0^t \|G_x(t-s)\|_1 \|b(s)\|_{p,\infty} \|u^{1+k}(s)\|_{p',1} ds \\
			\big[ \text{Lemma \ref{Lem:triv_estim_p_p'}}\big]&\leq \|G(t) \ast u_0 - u_0\|_1 + C \|b\|_{p,\infty} r^{k+1/p}  \int_0^t \frac{1}{\sqrt{t-s}} ds,
\end{split}
	\end{equation*}
	which converges to $0$ uniformly once $u_0$ is fixed. Hence, by Lemma \ref{Lem:conti_integr} above, $t \mapsto \Phi[u](t)$ is uniformly continuous in $L^1$.
	
We next prove that $\Phi(S) \subset S$: indeed using again Theorem \ref{Theo:holder_lorentz} in the first inequality, it holds
\begin{equation*}
    \begin{split}
        \|\Phi[u](t)\|_\infty &\le \|u_0\|_\infty + \int_0^t \|G_x(t-s)\|_{p',1} \|b(s) u^{k+1}(s)\|_{p,\infty}\,d s \\
 &\le \|u_0\|_\infty + \int_0^t \|G_x(t-s)\|_{p',1} \|b(s)\|_{p,\infty} \|u(s)\|_{\infty}^{k+1}\,d s \\
 \big[ \text{Proposition }\ref{prop:appx} \big] \quad &\le \|u_0\|_\infty + C t^{\frac{1}{2p'}} \underset{s\in[0,t]}{\textnormal{ess-sup}}\, \|u(s)\|^{k+1}_\infty \\
\big[ 2 \|u_0\|_\infty \leq r \big] \quad        &\le \frac{r}{2} + C \bar t^{\frac{1}{2p'}} r^{k+1}.
    \end{split}
\end{equation*}
Taking $\bar t \ll 1$ it holds that $\Phi(S) \subset S$.

Finally the same computations show that $\Phi[u]$ is a contraction in $C([0,\bar t],L^{p',1}(\R))$:
\begin{equation*}
        \begin{split}
        	\|\Phi[u](t) - \Phi[v](t)\|_{p',1} &\leq C\int_0^t \|G_x(t-s)\|_{p',1} \|b(s)\|_{p,\infty} \|u^{1+k}(s) - v^{1+k}(s)\|_{p',1} ds \\
        	&\leq C \|b\|_{p,\infty} t^{\frac{1}{2p'}} r^k \|u - v\|_{C_t L^{p',1}_x}, 
        \end{split}
\end{equation*}
so that for $\bar t \ll 1$ the statement follows. The case $p = \infty$ follows by H\"older's inequality because $\|b\|_{\infty,\infty} = \|b\|_\infty$, and gives a contraction in $C_t L^1_x$.
\end{proof}

For the second case, i.e. Conditions \eqref{eq:con2}, we start by proving that the second moment $M_2(u(t))$ is bounded if the solution $u$ of \eqref{m:eq:pde} belongs to $L^\infty_{t,x}$.

\begin{lemma}
	\label{Lem:first_moment}
If $\|u\|_{L^\infty_{t,x}} \leq r$, then
\begin{equation*}
\int (1+x^2) |u(t)| dx \leq \bigg( \int (1+x^2) |u_0| dx \bigg) e^{C(1 + r^k) t}.
\end{equation*}
\end{lemma}

\begin{proof}
Using the estimates by Corollary \ref{Cor:bound_b_infty}
\begin{equation*}
\begin{split}
&|b(x)| \le \|b_x\|_{p,\infty} x^{{1/p'}} + |b(0)| \leq C \sqrt{1+x^2} \le C(1+|x|),
\end{split}
\end{equation*}
by \eqref{m:eq:pde} it holds
\begin{equation*}
\begin{split} \frac{d}{dt} \int (1+x^2) u \,dx &= \int (1+x^2) (u_x-bu^{k+1})_x dx \\
        &= - 2 \int x u_x dx + \int 2 x bu^{k+1} dx \\
        &\le 2 \int u\,dx + 2 \int |x b|u^{k+1}\,dx \\
        &\le 2 \|u\|_1 + C \int (1+x^2) u^{k+1}\,dx\\
        &\le C (1 + \|u\|_\infty^k) \int (1+x^2) |u| dx.
\end{split}
\end{equation*}
Since $\|u\|_\infty \leq r$, one integrates the differential inequality to obtain the statement.
\end{proof}

\begin{corollary}
\label{Cor:cont_L_infty_case_2}
If $u(t) \in L^\infty$ and $M_2(u_0) < \infty$, then $t \mapsto \Phi[u](t)$ is uniformly continuous in the interval $t \in [t_0,t_1]$ with $t_0 > 0$, 
\end{corollary}

\begin{proof}
Following the same line as in Lemma \ref{Lem:conti_integr} above,
\begin{equation*}
\begin{split}
\|&u(t + \delta) - u(t)\|_\infty \\
&\leq \|G(t + \delta) - G(t)\|_1 \|u_0\|_\infty + \int_0^\delta \|G_x(\delta - s)\|_{2} \|b(t+s) u^{1+k}(t+s)\|_{2} ds \\
& \quad + \int_0^t \|G_x(t+\delta-s) - G_x(t-s)\|_{2} \|b(s) u^{1+k}(s)\|_{2} ds \\
&\leq \|G(t + \delta) - G(t)\|_1 \|u_0\|_\infty + C \Big(  \|u^{1+k}\|_{L^2_{t,x}} + \sup_{[t_0,t_1]} E(u(t))^{\frac{1}{2}} \|u\|_\infty^{\frac{1}{2} + k} \Big) \delta^{\frac{1}{4}} \\
& \quad + C \Big(  \|u^{1+k}\|_{2} + \sup_{[t_0,t_1]} E(u(t))^{\frac{1}{2}} \|u\|_\infty^{\frac{1}{2} + k} \Big) \int_0^t \|G_x(t + \delta - s) - G_x(t-s)\|_{2} ds,
\end{split}
\end{equation*}
where we recall that $E(u)$ is the second moment of $u$. As in the previous case, the r.h.s. converges to $0$ uniformly for $t \geq t_0$.
\end{proof}

\begin{proposition}
\label{prop:loc:exs:2} 
Assume that $b(t,x=0)$ is uniformly bounded, $b_x \in L^\infty_t L^{p,\infty}_x$, $k \geq 0$ and $u_0 \in L^\infty \cap L^1(\R)$ with bounded second order moment $E(u_0)$. Then $u \mapsto \Phi[u]$ is a contraction in the set
$$
S = \Big\{ u \in L^\infty([0,\bar t] \times \R) \cap C([0,\bar t],L^{2}(\R)), \|u\|_\infty \leq r, u(0) = u_0 \Big\},
$$
with $\bar t$ sufficiently small and $r > 2 \|u_0\|_\infty$. In particular, there is only one bounded solution for \eqref{m:eq:pde}, and it belongs to the space $C([0,\bar t], L^{2}(\R))$.
\end{proposition}

\begin{proof}
Following the same line of the proof of Proposition \ref{prop:loc:exs:1} we study first the continuity for $t \searrow 0$: 
	\begin{equation*}
\begin{split}
			\|u(t) - u_0\|_1 &\leq \|G(t) \ast u_0 - u_0\|_1 + \int_0^t \|G_x(t-s)\|_1 \|b(s) u^{1+k}(s)\|_1 ds \\
&\leq \|G(t) \ast u_0 - u_0\|_1 + C\int_0^t \|G_x(t-s)\|_1 \big[ (\|b(s,0) \|_\infty + 1) \|u^{k}(s)\|_\infty + \|u^{k}(s)\|_{\infty} E(u(s)) \big] ds \\
			&\leq \|G(t) \ast u_0 - u_0\|_1 + C r^k \big[ (\|b(x=0)\|_{\infty}+1) + \sup_{[0,t]}E(u_0) \big] \sqrt{t},
\end{split}
	\end{equation*}
which converges to $0$ as $t \to 0$. Hence, together with Corollary \ref{Cor:cont_L_infty_case_2} we obtain the uniform continuity in every $L^p$-space.

Next,
\begin{equation*}
    \begin{split}
        \|\Phi[u](\bar t)\|_\infty &\le \|u_0\|_\infty + \int_0^{\bar t} \|G_x(\bar t-s)\|_{2} \|b(s) u^{k+1}(s)\|_{2}\,d s \\
 &\le \|u_0\|_\infty + C\big( \|b(x=0)\|_{L^\infty_t}+1+ \sup_{[0,\bar t]}E(u(t))^{\frac{1}{2}} \big) \|u\|_{\infty}^{k+\frac{1}{2}} \int_0^{\bar t} \|G_x(\bar t-s)\|_{2} d s \\
&\le \|u_0\|_\infty + C (1+e^{C \bar t(1+r^k)/2 })r^{k + \frac{1}{2}}  \bar t^{\frac{1}{4}} \\
&\le \frac{r}{2} + C (1 +e^{C\bar t(1+r^k)/2 })r^{k + \frac{1}{2}}  \bar t^{\frac{1}{4}},
    \end{split}
\end{equation*}
Taking $\bar t \ll 1$ it holds that $\Phi(S) \subset S$.

Finally, for positive solutions,
\begin{equation*}
        \begin{split}
        	\int (1 + x^2) \big( \Phi[u](\bar t,x) - \Phi[v](\bar t,x) \big) &\leq \int_0^{\bar t} \|G_x(\bar t-s)\|_{2} \|b(s)(u^{1+k}(s) - v^{1+k}(s))\|_{1} ds \\
        	&\leq C \Big( (\|b(0)\|_{L^\infty_t}+1) (\|u\|_{L^\infty_{t,x}}^{k - \frac{1}{2}} + \|v\|_{L^\infty_{t,x}}^{k - \frac{1}{2}}) \Big)\|u - v\|_{C_t L^2_x} \bar t^{\frac{1}{4}} \\
& \quad + \big( E(u)^{1/2} \|u\|^{k - \frac{1}{2} }_{L^\infty_{t,x}} + E(v)^{1/2} \|v\|^{k - \frac{1}{2}}_{L^\infty_{t,x}} \big) \|u - v\|_{C_t L^2_x} \bar t^{\frac{1}{4}}, 
        \end{split}
\end{equation*}
so that for $\bar t \ll 1$ $\Phi$ is a contraction.
\end{proof}

\begin{remark}
    The condition $E(u_0) < +\infty$ that we used in Proposition \ref{prop:loc:exs:2} will not play a role in the rest of the paper, in the sense that the estimates obtained are independent on $E(u_0)$. The same can be said for the condition $k \geq \frac{1}{2}$. Clearly for the blow-up it is more interesting to study the PDE for large $k$. 
\end{remark}

\section{Global existence in the subcritical and critical case}
\label{s:glob}

In order to prove the global existence we consider a standard rescaling about a possible blow-up at point $(T,\hat x)$, where w.l.o.g. we assume that $\hat x=0$. We will show that the rescaled solutions decrease to $0$ with the appropriate speed in $L^2$-norm at time $T$ in the critical and subcritical case, so that the original unscaled solution is bounded by using the estimates contained in \cite{guidolin2022global}.

\subsection{Rescaled variables}
\label{Ss:realed_variables}

Set
\begin{equation*}
t = T(1 - e^{-\tau}),
\end{equation*}
and define 
\begin{equation*}
\begin{split}
    v(\tau,y) &= \sqrt{T} e^{-\tau/2} u\Big(T(1-e^{-\tau}),\sqrt{T e^{-\tau}} y\Big), \\
    \tilde b(\tau,y) &= (T e^{-\tau})^{\frac{1-k}{2}} b\Big(T(1-e^{-\tau}),\sqrt{Te^{-\tau}} y\Big).
\end{split}
\end{equation*}
The rescaled equation for $v$ is then
\begin{equation*}
v_\tau + \frac{1}{2} (yv)_y + (\tilde b v^{1+k})_y = v_{yy}, \quad v(\tau = 0) = \sqrt{T} u_0(\sqrt{T} y).
\end{equation*}
We observe that 
\begin{equation*}
\|\tilde b(\tau)\|_{p,\infty} = (Te^{-\tau})^{\frac{1 - k - 1/p}{2}} \big\| b(T(1-e^{-\tau})) \big\|_{p,\infty},
\end{equation*}
hence in particular in the critical and subcritical case it holds
\begin{equation}
\label{Equa:tildeb_Lp:subcr}
    \sup_{\tau > 0} \| \tilde b(\tau) \|_{p,\infty} \le T^{\frac{1 - 1/p - k}{2}} \sup_{t \in (0,T)} \|b(t)\|_{p,\infty} \leq C T^{\frac{1 - 1/p - k}{2}}, \qquad \text{when} \quad k \le 1 - \frac{1}{p} = \frac{1}{p'}.
\end{equation}
In the same way, for $b_x \in L^{p,\infty}(\R)$ 
\begin{equation}
\label{Equa:tildeby_Lp}
\|\tilde b_y(\tau)\|_{p,\infty} = (Te^{-\tau})^{\frac{2-k-1/p}{2}} \big\| b(T(1-e^{-\tau})) \big\|_{p,\infty},
\end{equation}
and then
\begin{equation}
\label{Equa:tildeby_Lp:subcr}
   \sup_{\tau > 0} \|\tilde b_y(\tau)\|_{p,\infty} \le T^{\frac{2 - 1/p - k}{2}} \sup_{t \in (0,T)} \|b_x(t)\|_{p,\infty} \leq C T^{\frac{2 - 1/p - k}{2}}, \qquad \text{for} \quad k \le 2 - 1/p.
\end{equation}
Moreover we observe that
\begin{equation}
    \label{Equa:L2_v_u}
    \|v(\tau)\|_2^2 = \sqrt{T} e^{- \frac{\tau}{2}} \|u(T(1- e^{-\tau}))\|_2^2.
\end{equation}

\subsection{Entropy dissipation}
\label{m:Ss:entropy:comp}

If $\eta:\R\to\R$ is a smooth $C^{1,1}$-function, then
\begin{equation}\label{eq:der:entropy}
\begin{split}
\frac{d}{d\tau} \int \eta(v) &= \int \eta' v_\tau = \int \eta'(v)  v_{yy} + \int \eta' \bigg( - \frac{1}{2} y v -{\hb}v^{k+1}\bigg)_y.
\end{split}
\end{equation}
Here we consider the entropy $\eta_a$ given by 
\begin{equation*}
    \eta_a(v)=\begin{cases}
        {v^2/2} & 0 \leq v \le a, \\
        a\left(v-{a}/{2}\right) & a < v < \infty,
    \end{cases}
\end{equation*}
and denote
$$
v_a = v \wedge a.
$$
The parameter $a$ will be chosen later on to be sufficiently small.

\begin{lemma}
\label{lemma:entropty}
Assume the exponent are critical or subcritical, i.e. Conditions \eqref{eq:con1} with $k\le 1-\frac{1}{p}$ or Conditions \eqref{eq:con2} and $k \le 2-\frac{1}{p}$. Then for every time $\bar \tau$ there exists a constant $a = a(p,k,\frac{e^{\bar \tau}}{T})$ such that 
it holds 
\begin{equation}
\label{Equa:entropy_decrease_1}
    \frac{d}{d\tau} \int \eta_{a}(v(\tau,y))\,dy\le-\frac{\|v_{a,y}(\tau)\|^2_2}{2} - \frac{\|v_a(\tau)\|_2^2}{8}
\end{equation}
for $\tau \geq \bar \tau$. If $k = 1 - \frac{1}{p}$ in \eqref{eq:con1} or $k = 2 - \frac{1}{p}$ in \eqref{eq:con2} then $a = a(p,k)$ is independent on $\bar \tau,T$.
\end{lemma}

\begin{proof}
By integration by parts, Equation \eqref{eq:der:entropy} becomes
\begin{equation}
\label{Equa:entropy_v_a}
    \frac{d}{d\tau} \int \eta_{a}(v) dy =-\|v_{a,y}\|^2_2-\frac{1}{4}\|v_{a}\|^2_2+\int v_{a,y}\hb v_a^{k+1}.
\end{equation}

Assume first Conditions \eqref{eq:con1}, and 
let 
\begin{equation*}
\frac{1}{p} + \frac{1}{p'} = \frac{1}{2}, \quad 2 < p < \infty. 
\end{equation*}
Fixed $p > 2$, choose $2 < \tilde p' < p'$ such that $(1 + k) \tilde p' > 2$ (so the constants below do not depend on $\tilde p'$)
and compute 
\begin{equation}
\label{m:eq:f3}
\begin{split}
\bigg| \int v_{a,y} \hb v_a^{k+1} \bigg| &\le  \|\hb v_a^{k+1}\|_2 \|v_{a,y}\|_2 \\
\big[ \text{Proposition \ref{prop:equiv_L_p}} \big] \quad 
&\leq \| \tilde b v_a^{k+1} \|_{2,2} \|v_{a,y}\|_2 \\
\big[ \text{Theorem \ref{Theo:holder_lorentz}} \big] \quad 
&\le C \|\hb\|_{p,\infty} \|v_a^{k+1}\|_{p',2} \|v_{a,y}\|_2\\
\big[ \text{(\ref{Equa:estima_scal}) and (\ref{Equa:tildeb_Lp:subcr})} \big] \quad 
&\leq C(p) (T e^{-\tau})^{\frac{1 - \frac{1}{p} - k}{2}} \|v_a\|_{(k+1)p',(k+1)2}^{k+1} \|v_{a,y}\|_2 \\
\big[ \text{Proposition \ref{prop:inclusion_lorentz} and Lemma \ref{Lem:triv_estim_p_p'}, $p' > \tilde p' >2$} \big] \quad 
&\leq C(p) (T e^{-\tau})^{\frac{1 - \frac{1}{p} - k}{2}} \|v\|_\infty^{(k+1) (1 - \frac{\tilde p'}{p'})} \\
& \qquad \cdot \|v_a\|_{(k+1) \tilde p',(k+1) \tilde p'}^{(k+1) \frac{\tilde p'}{p'}} \|v_{a,y}\|_2 \\
\big[ \text{Proposition \ref{prop:equiv_L_p}} \big] \quad 
&\leq C(p) T^{\frac{1 - \frac{1}{p} - k}{2}} \|v\|_\infty^{(k+1) (1 - \frac{\tilde p'}{p'})} \|v_a\|_{(k+1)\tilde p'}^{(k+1) \frac{\tilde p'}{p'}} \|v_{a,y}\|_2 \\
 \Big[ \|v_a\|_{(k+1) \tilde p'} \leq C(p,k) \|v_{a,y}\|_{2}^{\frac{1}{2} - \frac{1}{(k+1) \tilde p'}} \|v_a\|_{2}^{\frac{1}{2} + \frac{1}{(k+1) \tilde p'}} \Big] \quad
 &\le C(p,k) (T e^{-\tau})^{\frac{1 - \frac{1}{p} - k}{2}} \|v\|_\infty^{(k+1) (1 - \frac{\tilde p'}{p'})} \\
& \qquad \cdot \|v_{a,y}\|_2^{(\frac{k+1}{2} - \frac{1}{\tilde p'}) \frac{\tilde p'}{p'} + 1} \|v_a\|_2^{(\frac{k+1}{2} + \frac{1}{\tilde p'}) \frac{\tilde p'}{p'}} \\
\Big[\|v_a\|_\infty\le \|v_a\|_1^\frac{1}{3}\|v_{a,y}\|_2^\frac{2}{3}\Big]\quad&\le C(p,k) (T e^{-\tau})^{\frac{1 - \frac{1}{p} - k}{2}} \|v_{a,y}\|_2^{(k+1)(\frac{2}{3} -\frac{1}{6} \frac{\tilde p'}{p'})-\frac{1}{p'} + 1} \\
& \qquad \cdot  \|v_a\|_2^{(\frac{k+1}{2} + \frac{1}{\tilde p'}) \frac{\tilde p'}{p'}}\\ 
\bigg[ \alpha \beta \le \frac{\alpha^\gamma}{2}+2^{\frac{1}{\gamma-1}}\beta^\frac{\gamma}{\gamma-1},
\gamma = \frac{2}{(k+1)(\frac{2}{3} -\frac{1}{6} \frac{\tilde p'}{p'})-\frac{1}{p'}+1 }
\bigg] 
&\leq \frac{1}{2} \|v_{a,y}\|_2^2 + C(p,k) (T e^{-\tau})^{\frac{1 - \frac{1}{p} - k}{1-(k+1)(\frac{2}{3} -\frac{1}{6} \frac{\tilde p'}{p'})+\frac{1}{p'}}} \\
& \qquad \cdot \|v_a\|_2^{2 \frac{\frac{k+1}{2} \frac{\tilde p'}{p'} + \frac{1}{p'}}{1-(k+1)(\frac{2}{3} -\frac{1}{6} \frac{\tilde p'}{p'})+\frac{1}{p'}}} \\
\big[ \|v_a\|_2^2\le \|v_a\|_\infty \|v_a\|_1\le a \big] \quad
&\leq \frac{1}{2} \|v_{a,y}\|_2^2 + C(p,k)(T e^{-\tau})^{\frac{1 - \frac{1}{p} - k}{1-(k+1)(\frac{2}{3} -\frac{1}{6} \frac{\tilde p'}{p'})+\frac{1}{p'}}}\\
& \qquad \cdot a^\frac{\frac{k+1}3(2+\frac{\tilde p'}{p'})-1}{1-(k+1)(\frac{2}{3} -\frac{1}{6} \frac{\tilde p'}{p'})+\frac{1}{p'}} \|v_a\|_2^{2} \\
\bigg[ a = \frac{1}{\bar C(p,k)} \bigg( \frac{e^{\bar \tau}}{T} \bigg)^{\frac{1 - \frac{1}{p} - k}{\frac{k+1}3(2+\frac{\tilde p'}{p'})-1}} 
\bigg] \quad &\leq \frac{1}{2} \|v_{a,y}\|_2^2 + \frac{1}{8} \|v_a\|_2^{2}.
\end{split}
\end{equation}
Substituting \eqref{m:eq:f3} into \eqref{Equa:entropy_v_a} we get \eqref{Equa:entropy_decrease_1}. \\
The case $p = \infty$, i.e. $b \in L^\infty_{t,x}$, can be treated in a much simpler way, because $\|\hb\|_{\infty,\infty} = \|b\|_\infty$: by following the same lines as above
\begin{equation*}
\begin{split}
\bigg| \int v_{a,y} \hb v_a^{k+1} \bigg| 
&\le \|\hb\|_{\infty} \|v_a^{k+1}\|_2 \|v_{a,y}\|_2 \\
&\leq C (T e^{-\tau})^{\frac{1 - k}{2}} \|v_a\|_{(k+1)2}^{k+1} \|v_{a,y}\|_2 \\
 \big[ 
\|v_a\|_{2(k+1)} \leq C(k) \|v_{a,y}\|_{2}^{\frac{1}{2} - \frac{1}{2(k+1)}} \|v_a\|_{2}^{\frac{1}{2} + \frac{1}{2(k+1)}} \big] \quad
&\le C(k) (T e^{-\tau})^{\frac{1 - k}{2}} \|v_{a,y}\|_2^{1 + \frac{k}{2}} \|v_a\|_2^{1 + \frac{k}{2}} \\
\bigg[ \alpha \beta \le \frac{1}{2}\alpha^\gamma+2^{\frac{1}{\gamma-1}}\beta^\frac{\gamma}{\gamma-1},
\gamma = \frac{2}{1 + \frac{k}{2}}
\bigg] \quad
&\leq \frac{1}{2} \|v_{a,y}\|_2^2 + C(k) (T e^{-\tau})^{\frac{1 - k}{1 - \frac{k}{2}}} \|v_a\|_2^{2 \frac{1 + \frac{k}{2}}{1 - \frac{k}{2}}} \\
&\leq \frac{1}{2} \|v_{a,y}\|_2^2 + C(p,k) (T e^{-\tau})^{\frac{1 - k}{1 - \frac{k}{2}}} a^{\frac{2k}{1 - \frac{k}{2}}} \|v_a\|_2^{2} \\
\bigg[
a = \frac{1}{\bar C(k)} \bigg( \frac{e^{\bar \tau}}{T} \bigg)^{\frac{1 - k}{2k}}  \bigg] \quad &\leq \frac{1}{2} \|v_{a,y}\|_2^2 + \frac{1}{8} \|v_a\|_2^{2}.
\end{split}
\end{equation*}


Assume instead \eqref{eq:con2} and $k \leq 2 - \frac{1}{p}$: using here $\frac{1}{p} + \frac{1}{p'} = 1$, $1 < p < \infty$, and a fixed $1 < \tilde p' < p'$ such that $(2 + k) \tilde p' > 2$, we have 

\begin{equation}
\label{m:eq:f2}
\begin{split}
\bigg| \int v_{a,y} \hb v_a^{k+1} \bigg| &\leq \frac{1}{k+2} \bigg| \int \hb_y v_a^{k+2} \bigg| \\
\big[ \text{Theorem  \ref{Theo:holder_lorentz}} \big] \quad
&\leq  C(k)\|\hb_y\|_{p,\infty} \|v_a^{k+2}\|_{p',1} \\
\big[ \text{(\ref{Equa:tildeby_Lp}),(\ref{Equa:estima_scal})} \big] \quad
&\leq C(p,k) (T e^{-\tau})^{\frac{2 - \frac{1}{p} - k}{2}} \|v_a\|_{(k+2)p',k+2}^{k+2} \\
\big[ \text{Proposition \ref{prop:inclusion_lorentz} and Lemma \ref{Lem:triv_estim_p_p'}, $1 < \tilde p' < p'$} \big] \quad
&\leq C(p,k) (T e^{-\tau})^{\frac{2 - \frac{1}{p} - k}{2}} \|v\|_\infty^{(k+2) (1 - \frac{\tilde p'}{p'})} \|v_a\|_{(k+2) \tilde p'}^{(k+2) \frac{\tilde p'}{p'}} \\
\Big[ 
\|v_a\|_{(k+2) \tilde p'} \le C(p,k) \|v_{a,y}\|_{2}^{\frac{1}{2} - \frac{1}{(k+2) \tilde p'}} \|v_a\|_{2}^{\frac{1}{2} + \frac{1}{(k+2) \tilde p'}} \Big] \quad 
&\leq C(p,k) (T e^{-\tau})^{\frac{2 - \frac{1}{p} - k}{2}} \|v\|_\infty^{(k+2) (1 - \frac{\tilde p'}{p'})} \\
& \qquad \cdot \|v_{a,y}\|_2^{(\frac{k+2}{2} - \frac{1}{\tilde p'}) \frac{\tilde p'}{p'}} \|v_a\|_2^{(\frac{k+2}{2} + \frac{1}{\tilde p'}) \frac{\tilde p'}{p'}} \\
\big[\|v_a\|_\infty\le\|v_a\|_1^{\frac{1}{3}}\|v_{a,y}\|_2^{\frac{2}{3}}\big]\quad&\leq C(p,k) (T e^{-\tau})^{\frac{2 - \frac{1}{p} - k}{2}} \|v_{a,y}\|_2^{(k+2) (\frac{2}{3} - \frac{1}{6}\frac{\tilde p'}{p'})-\frac{1}{p'}} \\
& \qquad \cdot \|v_a\|_2^{(\frac{k+2}{2} + \frac{1}{\tilde p'}) \frac{\tilde p'}{p'}} \\
\bigg[ \alpha \beta \le \frac{1}{2} \alpha^\gamma+2^{\frac{1}{\gamma-1}}\beta^\frac{\gamma}{\gamma-1}, \gamma = \frac{2}{(k+2) (\frac{2}{3} - \frac{1}{6}\frac{\tilde p'}{p'})-\frac{1}{p'}} 
\bigg] \quad
&\leq \frac{1}{2} \|v_{a,y}\|_2^2 + C(p,k) (Te^{-\tau})^{\frac{2 - \frac{1}{p} - k}{2-(k+2) (\frac{2}{3} - \frac{1}{6}\frac{\tilde p'}{p'})+\frac{1}{p'}}} \\
& \qquad \cdot \|v_a\|_2^{2 \frac{\frac{k+2}{2} \frac{\tilde p'}{p'} + \frac{1}{p'}}{2-(k+2) (\frac{2}{3} - \frac{1}{6}\frac{\tilde p'}{p'})+\frac{1}{p'}}} \\
[\|v\|_2^2\le\|v\|_\infty\|v\|_1]\quad&\leq \frac{1}{2} \|v_{a,y}\|_2^2 + C(p,k) (T e^{-\tau})^{\frac{2 - \frac{1}{p} - k}{2-(k+2) (\frac{2}{3} - \frac{1}{6}\frac{\tilde p'}{p'})+\frac{1}{p'}}} \\
& \qquad \cdot a^{\frac{\frac{k+2}{3}(2+\frac{\tilde p'}{p'})-2}{2-(k+2) (\frac{2}{3} - \frac{1}{6}\frac{\tilde p'}{p'})+\frac{1}{p'}}} \|v_a\|_2^{2} \\
\bigg[ a \leq \frac{1}{\bar C(p,k)} \bigg( \frac{e^{\tau}}{T} \bigg)^{\frac{2 - \frac{1}{p} - k}{\frac{k+2}{3}(2+\frac{\tilde p'}{p'})-2}} \bigg] \quad &\leq \frac{1}{2} \|v_{a,y}\|_2^2 + \frac{1}{8} \|v_a\|_2^{2}.
\end{split}
\end{equation}
so that \eqref{Equa:entropy_decrease_1} follows as in the other cases. \\
Again the case $p = \infty$, $k > 0$ can be obtained by applying H\"older's inequality:
\begin{equation*}
\begin{split}
\bigg| \int v_{a,y} \hb v_a^{k+1} \bigg| &\leq \frac{1}{k+2} \bigg| \int \hb_y v_a^{k+2} \bigg| \\
&\leq C(k) \|\hb_y\|_{\infty} \|v_a^{k+2}\|_{1} \\
&\leq C(k) (T e^{-\tau})^{\frac{2 - k}{2}} \|v_a\|_{k+2}^{k+2} \\
\big[ 
\|v_a\|_{k+2} \le C(k) \|v_{a,y}\|_{2}^{\frac{1}{2} - \frac{1}{k+2}} \|v_a\|_{2}^{\frac{1}{2} + \frac{1}{k+2}} \big] \quad 
&\leq C(k) (T e^{-\tau})^{\frac{2 - k}{2}} \|v_{a,y}\|_2^{\frac{k}{2}} \|v_a\|_2^{2 + \frac{k}{2}} \\
\bigg[ \alpha \beta \le \frac{1}{2} \alpha^\gamma+2^{\frac{1}{\gamma-1}} \beta^\frac{\gamma}{\gamma-1}, 
\gamma = \frac{4}{k} 
\bigg] \quad
&\leq \frac{1}{2} \|v_{a,y}\|_2^2 + C(k) (T e^{-\tau})^{\frac{2 - k}{2 - \frac{k}{2}}} \|v_a\|_2^{2 \frac{2 + \frac{k}{2}}{2 - \frac{k}{2}}} \\
&\leq \frac{1}{2} \|v_{a,y}\|_2^2 + C(k) (T e^{-\tau})^{\frac{2 - k}{2 - \frac{k}{2}}} a^{\frac{2k}{2 - \frac{k}{2}}} \|v_a\|_2^2 \\
\bigg[ a = \frac{1}{\bar C(k)} \bigg( \frac{e^{\bar \tau}}{T} \bigg)^{\frac{2 - k}{2k}} \bigg] \quad
&\leq \frac{1}{2} \|v_{a,y}\|_2^2 + \frac{1}{8} \|v_a\|_2^{2},
\end{split}
\end{equation*}
This concludes the proof.
\end{proof}

Set now
\begin{equation}
\label{Equa:time}
e^{\bar \tau} = T, \quad \bar a = a \bigg( p,k,\frac{e^{\bar \tau}}{T} \bigg) = a \big(p,k,1),
\end{equation}
where $a(\bar \tau,k,p)$ is given by the previous lemma: notice that it is independent on $T$. 

\begin{lemma}
\label{lemma:entropty2}
Assume \eqref{eq:con1} and $0 < k \le 1-\frac{1}{p}$ or \eqref{eq:con2} and $0 < k \le 2-\frac{1}{p}$. Then it holds 
\begin{equation*}
\int_{\bar \tau}^\infty \bigg( \frac{\|v_{\bar a}(\tau)\|_\infty^3}{2} + \frac{\|v_{\bar a}(\tau)\|_2^2}{8} \bigg)\, d\tau \leq \frac{3\bar a}{2}.
\end{equation*}
\end{lemma}

\begin{proof}
We consider the case $b(t) \in L^{p,\infty}$, the other case being completely similar. By Lemma \ref{lemma:entropty} and Gagliardo-Nirenberg inequality we have for $\tau \geq \bar \tau$ that 
\begin{equation*}
\begin{split}
   \frac{d}{d\tau}  \int \eta_{\bar a} &\leq 
- \frac{1}{2} \|v_{\bar a,y}\|_2^2 - \frac{\|v_{\bar a}\|_2^2}{8} \\
\big[ 
\|v_{\bar a}\|_\infty \leq \|v_{\bar a,y}\|_2^{2/3} \|v_{\bar a}\|_1^{1/3} \big] \quad 
&\leq - \frac{\|v_{\bar a}\|_\infty^3}{2\|v_{\bar a}\|_1} - \frac{\|v_{\bar a}\|_2^2}{8}. 
\end{split}
\end{equation*}
Integrating and observing that
\begin{equation*}
    \|v_{\bar a}(\tau)\|_1 \leq \|v(\tau)\|_1 = \|u(T(1-e^{-\tau}))\|_1 = \|u_0\|_1 = 1,
\end{equation*}
\begin{equation*}
    \int \eta_{\bar a}(v(\tau,y)) dy \leq \frac{1}{2} \|v_{\bar a}(\tau)\|_2^2 + \bar a \|v(\tau)\|_1 \leq \frac{3}{2} \bar a \|u_0\|_1 = \frac{3\bar a}{2},
\end{equation*}
we obtain 
\begin{equation*}
\begin{split}
\int_{\bar \tau}^\infty \bigg( \frac{\|v_{\bar a}(\tau)\|^3}{2} + \frac{\|v_{\bar a}\|_2^2}{8} \bigg) d\tau &\leq \int \eta_{\bar a}(v(\bar \tau,y)) dy \leq \frac{3}{2} \bar a,
\end{split}
\end{equation*}
which is the statement.
\end{proof}

\subsection{Bounds on the $L^2$-norm}
\label{Ss:bounds_integral_norms}

With the results of the above section, we can estimate the $L2p$-norms of the solution. The quantities $\bar a$ and $\bar \tau$ are defined in \eqref{Equa:time}.

\begin{lemma}
    \label{Lem:decay_L_2_case_1}
    If $b$ satisfies Conditions \eqref{eq:con1} and $k \leq 1 - \frac{1}{p}$ or Conditions \eqref{eq:con2} and $k \leq 2 - \frac{1}{p}$, then it holds 
    \begin{equation*}
        \|v(\tau)\|_2^2 \leq \hat C(k,p) \bar a e^{- \frac{\tau - \bar \tau}{2}}.
    \end{equation*}
\end{lemma}

\begin{proof}
We consider only the case $b(t) \in L^{p,\infty}$ and $p < \infty$, since the other cases ($b(t) \in L^{p,\infty}$ or $b_x(t) \in L^{p,\infty}$) can be obtained with similar computations.

Taking $p' > \tilde p' \geq 2$, $(1 + k) \tilde p' > 2$, 
with the same computations as in \eqref{m:eq:f3} we obtain
\begin{equation*}
\begin{split}
\frac{d}{d\tau} \|v\|_2^2 &\leq -\frac{1}{2} \|v_{y}\|_2^2 + C(p,k) (T e^{-\tau})^{\frac{1 - \frac{1}{p} - k}{1-(k+1)(\frac{2}{3} -\frac{1}{6} \frac{\tilde p'}{p'})+\frac{1}{p'}}}  \|v\|_2^{2 \frac{\frac{k+1}{2} \frac{\tilde p'}{p'} + \frac{1}{p'}}{1-(k+1)(\frac{2}{3} -\frac{1}{6} \frac{\tilde p'}{p'})+\frac{1}{p'}}} \\ 
\big[ T e^{\bar \tau} = 1, \tau \geq \bar \tau \big] \quad &= \bigg( - \frac{1}{2} + C(p,k) \|v\|_2^{2 \frac{\frac{k+1}3 (2+\frac{\tilde p'}{p'}) - 1}{1-(k+1)(\frac{2}{3} -\frac{1}{6} \frac{\tilde p'}{p'})+\frac{1}{p'}}} \bigg) \|v\|_2^2.
\end{split}
\end{equation*}

The differential inequality
\begin{equation*}
    \frac{dz}{d\tau} \leq \bigg( - \frac{1}{2} + C(p,k) z^{\alpha} \bigg) z, \qquad \alpha > 0,
\end{equation*}
has a solution bounded solution if for some $\hat \tau$
$$
C(p,k) z(\hat \tau)^{\alpha} < \frac{1}{4}, 
$$
and in this case it holds
\begin{equation*}
    z(\tau) \le\frac{1}{\big(2C(p,k)+e^{\frac{\alpha}{2}(\tau-\hat\tau)}(\frac{1}{z^{\alpha}(\hat \tau)}-2C(p,k))\big)^{1/\alpha}}\leq  2^\frac{1}{\alpha}z(\hat \tau) e^{-\frac{\tau - \hat \tau}{2}}.
\end{equation*}

Using Lemma \ref{lemma:entropty2}, we estimate the first time $\tilde \tau$ such that
$$
\|v(\tilde \tau)\|_\infty \leq \bar a
$$
as follows
\begin{equation*}
\begin{split}
(\tilde \tau - \bar\tau) \bar a^3\le \int_{\bar \tau}^{\tilde \tau} \|v_{\bar a}\|_\infty^3 d\tau \leq \frac{3\bar a}{2}, \quad \tilde \tau - \bar\tau \leq \frac{3}{2 \bar a^2}.
\end{split}
\end{equation*}
Also, notice that by \eqref{m:eq:f3} and the choice of $\bar a$
\begin{equation*}
    C(p,k) \|v(\tilde \tau)\|_2^{2\alpha}\le C(p,k) \bar a^\alpha \leq \frac{1}{8} 
\end{equation*}
so that we can take $\hat \tau=\tilde \tau$. Hence
\begin{equation*}
        \|v(\tau)\|_2^2 \leq \hat C(p,k) \bar a e^{- \frac{\tau - \bar \tau}{2}}, \quad \tau\ge\bar \tau + \frac{3}{2 \bar a^2}. 
\end{equation*}
Hence the constant in the statement can be bounded by $\hat C(k,p) \leq 2^{\frac{1-(k+1)(\frac{2}{3} -\frac{1}{6} \frac{\tilde p'}{p'})+\frac{1}{p'}}{\frac{k+1}3 (2+\frac{\tilde p'}{p'}) - 1}} e^{\frac{3}{2 \bar a^2}}$.
\end{proof}

\begin{corollary}
    \label{Cor:bound_u_2}
    Assume Conditions \eqref{eq:con1} and $k \leq 1 - \frac{1}{p}$, or Conditions \eqref{eq:con2} and $k \leq 2 - \frac{1}{p}$: then it holds 
\begin{equation}
\label{Equa:clam_bound_u_1}
    \|u(t)\|_2^2 \le \hat C(p,k) \bar a, \qquad \forall t \in [T-1,T).
\end{equation}
\end{corollary}

\begin{proof}
Recalling that by \eqref{Equa:L2_v_u}
\begin{equation*}
    \|u(T(1-e^{-\tau}))\|_2^2 = \frac{e^{\tau/2}}{\sqrt{T}} \|v(\tau)\|_2^2,
\end{equation*}
Lemma \ref{Lem:decay_L_2_case_1} and \eqref{Equa:time} give for
$$
t = T(1 - e^{-\tau}) \geq T(1 - e^{-\bar \tau}) = T - 1
$$
that
$$
\|u(T(1-e^{-\tau}))\|_2^2 \leq \hat C(p,k) \bar a \frac{e^{\frac{\bar \tau}{2}}}{\sqrt{T}} \leq \hat C(p,k) \bar a.
$$
This proves \eqref{Equa:clam_bound_u_1}.
\end{proof}

\subsection{Bound on the $L^\infty$-norm}
\label{Ss:bounds_L_infty}

Having proved a uniform bound on the $L^2$-norm, now we are ready to bound the $L^\infty$-norm. This proves the first part of Theorem \ref{Theo:existence_1}

\begin{theorem}
\label{Theo:not_blow_up}
Assume Conditions \eqref{eq:con1} with $k\le 1-\frac{1}{p}$ or Conditions \eqref{eq:con2} with $k \leq 2 - \frac{2}{p}$. Then $u$ does not blow up at time $T$.
\end{theorem}

\begin{proof}

We do the proof in the first case, the second being completely similar.

We notice first that the same computations leading to \cite[Theorem 3.8]{guidolin2022global} gives
\begin{equation}
\label{Equa:guidolin}
\|u(T)\|_\infty \leq C \max \Big\{ \|u(\hat T)\|_\infty, \sup_{[\hat T,T]} \|u(t)\|_2^{\frac{1 - \frac{1}{{p}}}{1 - \frac{1}{{p}} - \frac{k}{2}}} \Big\}.
\end{equation}
The statement of the theorem is now a consequence of Corollary \ref{Cor:bound_u_2} with $\hat T = T-1$.

For completeness, we rewrite the proof of \cite{guidolin2022global} for the case $b \in L^\infty_t L^{p,\infty}_x$, since in that paper it is only considered the case $b_x(t) \in L^\infty$.

\medskip

\noindent {\it Step 1 .} Write for some $2 < \tilde p < p$ (if $p = \infty$ then $\tilde p = p$)
\begin{equation}
\label{Equa:first_guidolin}
    \begin{split}
        \frac{d}{dt} \int u^{2n} dx &= - \frac{2(2n-1)}{n} \int (u^n_x)^2 dx - 2n (2n-1) \int u^{2n-2} u_x b u^{1+k} dx \\
        &= - \frac{2(2n-1)}{n} \int (u^n_x)^2 dx - 2 (2n-1) \int u^n_x b u^{n+k} dx \\
        &\leq - \frac{2(2n-1)}{n} \|u^n_x\|_2^2 + 2 (2n-1) \| u^n_x\|_2 \|b u^{n+k}\|_2 \\
        &\leq - \frac{2n-1}{n} \|u^n_x\|_2^2 + (2n-1) n \|b u^{n+k}\|_2^2 \\
\bigg[ 1 = \frac{2}{p} + \frac{p-2}{p} \bigg] \quad        &\leq - \frac{2n-1}{n} \|u^n_x\|_2^2 + C n (2n-1) \|b\|_{p,\infty}^2 \|u^{n + k}\|_{\frac{2p}{p-2},2}^{2} \\
\big[ (\ref{Equa:estima_scal}) \big] \quad &\leq - \frac{2n-1}{n} \|u^n_x\|_2^2 + C(p) n (2n-1) \|b\|_{p,\infty}^2 \|u^n\|_{(1 + \frac{k}{n}) \frac{2p}{p-2},(1 + \frac{k}{n}) 2}^{2(1+ \frac{k}{n})}.     \end{split}
\end{equation}
For $2 < \tilde p < p <\infty$, use now Lemma \ref{Lem:triv_estim_p_p'} with exponents
\begin{equation*}
    p_1 = \bigg( 1 + \frac{k}{n} \bigg) 2 < \bar p = \bigg( 1 + \frac{k}{n} \bigg) \frac{2p}{p-2} < p_2 = \bigg( 1 + \frac{k}{n} \bigg) \tilde p' = \bigg( 1 + \frac{k}{n} \bigg) \frac{2 \tilde p}{\tilde p - 2},
\end{equation*}
\begin{equation*}
    q_1 = p_1 = \bigg( 1 + \frac{k}{n} \bigg) 2, \quad \bar q = q_1, \quad q_2 = p_2 = \bigg( 1 + \frac{k}{n} \bigg) \frac{2 \tilde p}{\tilde p - 2},
\end{equation*}
so that
\begin{equation*}
\begin{split}
    \|u^n\|_{(1 + \frac{k}{n}) \frac{2p}{p-2},(1 + \frac{k}{n}) 2} &\leq C(p,k,n) \|u^n\|_{(1 + \frac{k}{n})2,(1 + \frac{k}{n})2}^{1-\theta} \|u^n\|_{(1 + \frac{k}{n}) \frac{2\tilde p}{\tilde p -2},(1 + \frac{k}{n}) \frac{2\tilde p}{\tilde p -2}}^{\theta}, 
\end{split}
\end{equation*}
with
\begin{equation*}
    \theta = \frac{\tilde p}{p} \in (0,1),
\end{equation*}
\begin{equation*}
    \begin{split}
        C(p,k,n) &= \bigg( \frac{p^2}{2 (p - \tilde p)} \bigg)^{\frac{1}{(1 + \frac{k}{n}) 2}} \bigg( \frac{p}{p-2} \bigg)^{\frac{1-\theta}{(1 + \frac{k}{n}) \frac{2 \tilde p}{\tilde p -2}}} \bigg( \frac{\frac{p}{p-2}}{\frac{\tilde p}{\tilde p - 2}} \bigg)^{\frac{\theta}{(1 + \frac{k}{n}) 2}}.
    \end{split}
\end{equation*}
By Proposition \ref{prop:equiv_L_p} we conclude that
\begin{equation*}
\begin{split}
    \|u^n\|_{(1 + \frac{k}{n}) \frac{2p}{p-2},(1 + \frac{k}{n}) 2} 
\leq C(p,k,n) \|u^n\|_{(1 + \frac{k}{n})2}^{1-\theta} \|u^n\|_{(1 + \frac{k}{n}) \frac{2\tilde p}{\tilde p -2}}^{\theta},
\end{split}
\end{equation*}
for another constant $C(p,k,n)$ which is uniformly bounded w.r.t. $n$ once $p,\tilde p,k$ are fixed. We can thus continue \eqref{Equa:first_guidolin} as
\begin{equation*}
    \begin{split}
        \frac{d}{dt} \int u^{2n} dx
&\leq - \frac{2n-1}{n} \|u^n_x\|_2^2 + C n (2n-1) \|b\|_{p,\infty}^2 \|u^n\|_{(1 + \frac{k}{n})2}^{2 (1 + \frac{k}{n}) (1-\theta)} \|u^n\|_{(1 + \frac{k}{n}) \tilde p'}^{2 (1 + \frac{k}{n}) \theta}.
    \end{split}
\end{equation*}
For $p = \infty$ one obtains the same formula above with $\tilde p' = 2$ and $\theta = 1$.

\medskip

\noindent 
{\it Step 2.}
Hence $\|u\|_{2n}$ is increasing only if
\begin{equation*}
    \|u^n_x\|_2 \leq C n \|u^n\|_{(1 + \frac{k}{n})2}^{(1 + \frac{k}{n}) (1-\theta)} \|u^n\|_{(1 + \frac{k}{n}) \tilde p'}^{(1 + \frac{k}{n}) \theta}.
\end{equation*}
Then using the embeddings by Gagliardo-Niremberg with exponents
\begin{equation*}
    \frac{1}{(1 + \frac{k}{n}) \tilde p'} = \frac{\tilde p-2}{2\tilde p(1 + \frac{k}{n})} = - \frac{a}{2} + \frac{1-a}{1}, \quad a = \frac{2}{3} \bigg( 1 - \frac{\tilde p-2}{2\tilde p (1 + \frac{k}{n})} \bigg),
\end{equation*}
for the function $w = u^n$ we obtain
\begin{equation*}
\begin{split}
    \|w_x\|_2 &\leq C n \|u^n\|_{(1 + \frac{k}{n})2}^{(1 + \frac{k}{n}) (1-\theta)} \big( \|w_x\|_2^a \|w\|_1^{1-a} \big)^{(1 + \frac{k}{n}) \theta}.
\end{split}
\end{equation*}
Using similarly Gagliardo-Niremberg with
\begin{equation*}
    \frac{1}{(1 + \frac{k}{n}) 2} = - \frac{b}{2} + 1 - b, \quad b = \frac{2}{3} \bigg( 1 - \frac{1}{(1 + \frac{k}{n}) 2} \bigg),
\end{equation*}
we can write
\begin{equation*}
\begin{split}
    \|w_x\|_2 &\leq C n \big( \|w_x\|_2^b \|w\|_1^{1-b} \big)^{(1 + \frac{k}{n})(1 - \theta)} \big( \|w_x\|_2^a \|w\|_1^{1-a} \big)^{(1 + \frac{k}{n}) \theta} \\
    &= C n \|w_x\|_2^{(1 + \frac{k}{n}) [b (1 - \theta) + a \theta]} \|w\|_1^{(1 + \frac{k}{n}) [(1-b)(1 - \theta) + (1-a) \theta)]} \\
    &= C n \|w_x\|_2^{(1 + \frac{k}{n}) \frac{2}{3} (1 - \frac{p-2}{(1 + \frac{k}{n})2p})} \|w\|_1^{(1 + \frac{k}{n})(\frac{1}{3} + \frac{2}{3} \frac{p-2}{(1 + \frac{k}{n})2p})}.
\end{split}
\end{equation*}

This can be rewritten as
\begin{equation*}
    \|w_x\|_2^{\frac{2}{3}(1 - \frac{1}{p} - \frac{k}{n})} \leq  C n \|w\|_1^{\frac{2}{3}(1 - \frac{1}{p} + \frac{k}{2n})}.
\end{equation*}
Using again Gagliardo-Niremberg we get
\begin{equation*}
\begin{split}
\|w\|_2 &\leq C \|w_x\|^{\frac{1}{3}} \|w\|_1^{\frac{2}{3}} \leq C n^{\frac{\frac{1}{2}}{1 - \frac{1}{p} - \frac{k}{n}}} \|w\|_1^{1 + \frac{\frac{k}{2n}}{1 - \frac{1}{p} - \frac{k}{n}}},
\end{split}
\end{equation*}
which rewritten for $u$ becomes
\begin{equation}
\label{Equa:region_invariant}
    \|u\|_{2n} \leq C^{\frac{1}{n}} n^{ \frac{\frac{1}{2n}}{1 - \frac{1}{p} - \frac{k}{n}}} \|u\|_n^{1 + \frac{\frac{k}{2n}}{1 - \frac{1}{p} - \frac{k}{n}}}.
\end{equation}

\medskip

\noindent{\it Step 3.} The above estimate implies that 
\begin{equation*}
    \max_{t \in [t_0,t]} \|u(t)\|_{2n} \leq \max \bigg\{ \|u(t_0)\|_{2n}, C^{\frac{1}{n}} n^{ \frac{\frac{1}{2n}}{1 - \frac{1}{p} - \frac{k}{n}}} \sup_{t \in [t_0,t]} \|u\|_n^{1 + \frac{\frac{k}{2n}}{1 - \frac{1}{p} - \frac{k}{n}}} \bigg\},
\end{equation*}
because the solution $u$ is decreasing when \eqref{Equa:region_invariant} is not satisfied.

Iterating the procedure for $n = 2^m$, $k=1,\dots,M$, one obtains
\begin{equation*}
\begin{split}
    \max_{t \in [t_0,t]} \|u(t)\|_{2^M} \leq \max \bigg\{ &\|u(t_0)\|_{2^M}, \\
    & \dots \\
    & C^{\alpha(4,M)} 2^{\beta(4,M)} \|u(t_0)\|_4^{\gamma(4,M)}, \\
    & C^{\alpha(2,M)} 2^{\beta(2,M)} \max_{t \in [t_0,t]} \|u\|_2^{\gamma(2,M)} \bigg\}.
\end{split}
\end{equation*}
The constants $\alpha(M',M)$, $\beta(M',M)$, $\gamma(M',M)$ are computed by iterating the exponent of \eqref{Equa:region_invariant}:
\begin{equation*}
    \gamma(M',M) = \prod_{n = M'+1}^{M} \frac{1 - \frac{1}{p} - \frac{k}{2^{n}}}{1 - \frac{1}{p} - \frac{k}{2^{n-1}}} = \frac{1 - \frac{1}{p} - \frac{k}{2^M}}{1 - \frac{1}{p} - \frac{k}{2^{M'}}}, \quad \gamma(M,M) = 1,
\end{equation*}
\begin{equation*}
    \alpha(M',M) = \sum_{n = M'}^{M-1} 2^{-n} \gamma(n+1,M) \leq \bar \alpha,
\end{equation*}
\begin{equation*}
    \beta(M',M) = \sum_{n = M'}^{M-1} \frac{n 2^{-n-1}}{1 - \frac{1}{p} - \frac{k}{2^{n}}} \gamma(n+1,M) \leq \bar \beta.
\end{equation*}
By applying repeatedly H\"older's inequality
\begin{equation*}
\|u(t_0)\|_{2^{M'}}^{\gamma(M',M)} \leq \big( \|u(t_0)\|_{2^{M'+1}}^{\frac{2}{3}} \|u(t_0)\|_{2^{M'-1}}^{\frac{1}{3}} \big)^{\gamma(M',M)} = \|u(t_0)\|_{2^{M'+1}}^{\gamma(M'+1,M)} \|u(t_0)\|_{2^{M'-1}}^{\gamma(M'-1,M)},
\end{equation*}
we get
\begin{equation*}
\begin{split}
    \max_{t \in [t_0,t]} \|u(t)\|_{2^M} \leq C^{\bar \alpha} 2^{\bar \beta} \max \bigg\{ &\|u(t_0)\|_{2^M}, 
\max_{t \in [t_0,t]} \|u\|_2^{\gamma(2,M)} \bigg\}.
\end{split}
\end{equation*}
Letting $M \to \infty$ one recovers \eqref{Equa:guidolin}.
\end{proof}

We conclude with the following corollary, which follows by using \eqref{Equa:guidolin} in $[0,t]$ because the bound of Corollary \ref{Cor:bound_u_2} is uniform in $t$.

\begin{corollary}
    \label{Cor:coroallry_uniform}
    Assume that $\|b(t)\|_{p,\infty}$ under Conditions \eqref{Equa:con1} with $k \leq 1 - \frac{1}{p}$ (or $\|b_x(t)\|_{p,\infty}$ under conditions \eqref{Equa:con2} with $k \leq 2 - \frac{1}{p}$) is uniformly bounded for $t \in (0,\infty)$. Then $\|u(t)\|_\infty$ is uniformly bounded. 
\end{corollary}

\subsection{The critical case $k=1 - \frac{1}{p}$ or $k = 2 -\frac{1}{p}$}
\label{Ss:critical_case_integr}

In this section we study the case \eqref{Equa:con1} when $\|b(t)\|_{p,\infty}$ is uniform bounded in time and the exponent is critical, i.e. $k = 1 - \frac{1}{p}$, $p > 2$, or Conditions \eqref{Equa:con2} with $\|b_x(t)\|_{p,\infty}$ uniformly bounded in time and $k = 2 - \frac{1}{p}$, $p > 1$. We consider only the first case, being the analysis completely similar.

It follows form Lemma \ref{lemma:entropty2} that instead of \eqref{Equa:time} we can just choose $\bar \tau = 0$ and $\bar a = a(p,k)$ and then Lemma \ref{Lem:decay_L_2_case_1} gives
\begin{equation*}
    \|v(\tau)\|_2^2 \leq \hat C(k,p) \bar a e^{-\frac{\tau}{2}}. 
\end{equation*}
Using the definition of $v$ we obtain for $\tau \geq 0$
\begin{equation}
\label{Equa:decaycrit_u_2_2}
    \|u(T(1 - e^{-\tau}))\|_2 \leq \frac{C}{\sqrt{T}}.
\end{equation}
Letting $\tau \to \infty$ we obtain the decay of the $L^2$-norm as for the heat kernel.

Again Chebyshev's inequality applied to Lemma \ref{lemma:entropty2} gives that there exists $\hat \tau \in \frac{3}{2\bar a^2}[1,2]$ such that $\|v(\hat \tau)\|_\infty \leq \bar a$, and then
\begin{equation*}
\|u(T(1 - e^{\hat \tau}))\|_\infty = \frac{\|v(\hat \tau)\|_\infty}{\sqrt{T e^{- \hat \tau}}} \leq \frac{\bar a e^{\frac{3}{\bar a^2}}}{\sqrt{T}}.
\end{equation*}
Using again the estimate \eqref{Equa:guidolin} and noticing that for the critical case $k = 1 - \frac{1}{p}$
\begin{equation*}
\frac{\frac{k}{2}}{1 - \frac{1}{p} - \frac{k}{2}} = 2,
\end{equation*}
using \eqref{Equa:decaycrit_u_2_2}
we get for $\hat T = T (1 - e^{- \hat \tau}) \geq T (1 - e^{-\frac{3}{2\bar a^2}})$ \begin{equation*}
\begin{split}
\|u(T)\|_\infty &\leq C \max \Big\{ \|u(\hat T)\|_\infty, \sup_{\hat T,T} \|u(t)\|_2^2 \Big\} \\
&\leq C \max \bigg\{ \frac{1}{\sqrt{T}}, \frac{1}{\sqrt{\hat T}} \bigg\} \leq \frac{C}{\sqrt{T}}.
\end{split}
\end{equation*}

We thus have proved the following

\begin{theorem}
\label{Theo:decay_bx_critical}
If $b \in L^\infty_t L^{p,\infty}_x$, $p \in (1,\infty]$, and $k = 1 - \frac{1}{2}$, or $b_x \in L^\infty_t L^{p,\infty}_x$, $p \in (1,\infty]$ and $k = 2 - \frac{1}{p}$, then for a constant $C$ independent on $u_0$ it holds
\begin{equation*}
\|u(t)\|_\infty \leq \frac{C}{\sqrt{t}}.
\end{equation*}
\end{theorem}

\section{Blow-up in finite time}
\label{s:blowup}

In this section we show that above the critical exponent $k>1-\frac{1}{p}$ for the case \eqref{eq:con1} or $k>2-\frac{1}{p}$ for the case \eqref{eq:con2}, the solutions blows up in general: we will prove this statement for $b \in L^p(\R)$ in the first case and $b_x\in L^p(\R)$ in the second case. A similar result has already been proved in \cite{toscani2012finite}, we repeat it here for completeness.

\subsection{Case \eqref{eq:con1}}
\label{Ss:case_1_blowup}

Let $0 < \alpha < 1$, $\beta > \alpha$ and $k > 1 - \alpha$. Consider an integrable smooth function $b \in L^p$ satisfying
\begin{equation}
\label{eq:syst:b}
    - |x|^{-\alpha} \leq b(x) \leq -|x|^{- \alpha}
 \ind_{|x| \leq \bar x} - |x|^{- \beta} \ind_{|x| > \bar x}
\end{equation}
with $\alpha p < 1, \beta p > 1$, for some constant $\bar x$ such that
\begin{equation}
\label{Equa:bar_x_cond}
    1 \leq \bar x \leq \bigg( \frac{\beta + k -1}{\alpha + k -1} \bigg)^{\frac{k}{\beta - \alpha}}.
\end{equation}

\begin{proposition}
\label{lemma:func1}
Fix a positive measurable function $f : \R \to [0,\infty)$.
\begin{enumerate}
    \item If
\begin{equation}\label{eq:cond:1:c}
\left( \frac{k}{k - (1-\alpha)} \right)^{-\frac{k}{3k+1 + \alpha}} 
\left( \int f^{k+1} |xb| \right)^{-\frac 1 {3k + 1 + \alpha}} E^\frac{k+1}{3k + 1 + \alpha} 
\le \bar x,
    \end{equation}
    then  
    \begin{equation}
\label{Equa:cond_1_m_vs_E}
        m \leq 2 
        \left( \frac{k}{k-(1 - \alpha)} \right)^\frac{2k}{3k + 1 + \alpha} \left(\int f^{k+1} |xb| \right)^\frac{2}{3k + 1 + \alpha} E^\frac{k-(\alpha+1)}{3k + 1 + \alpha}.
    \end{equation}
    \item If
\begin{equation}\label{eq:cond:2:c}
        \left(\int f^{k+1} |xb| dx \right)^{- \frac{1}{3k+1-\beta}} \left( \frac{2k}{k - (1 \beta)} \right)^{-\frac{k}{3k + 1 + \beta}} E^{\frac{k+1}{3k+1-\beta}} \ge x,
    \end{equation}
    then 
    \begin{equation}
    \label{Equa:m_esimta_beta}
    m \le 2 \left( \frac{2k}{k -  (1 - \beta)} \right)^{\frac{2k}{3k + 1 + \beta}} \left( \int f^{k+1} |xb| dx \right)^{\frac{2}{3k + 1 + \beta}} E^{\frac{k - (1 - \beta)}{3k + 1 +\beta}}.
    \end{equation}
\end{enumerate}
\end{proposition}

\begin{proof}
{\it Case 1.} We compute
\begin{equation*}
    \begin{split}
        m &= \int f dx = \int_{-R}^R f dx + \int_{|x| > r} f dx \\
        &\leq \int_{-R}^R f |x|^{\frac{1-\alpha}{k+1}} \frac{1}{|x|^{\frac{1-\alpha}{k+1}}} dx + \frac{1}{R^2} \int_{|x| > R} |x|^2 f dx \\
        &\le \left( \int_{-R}^R f^{k+1} |x|^{1-\alpha} dx \right)^{\frac 1 {k+1}} \left( \int_{-R}^{R} \left(\frac{1}{|x|}\right)^{\frac{1-\alpha}{k}} \right)^{\frac{k}{k+1}} + \frac{1}{R^2} E\\
\big[ \text{if } R \leq \bar x \big] \quad       &\leq \left( \int_{\R} f^{k+1} |x b| \right)^{\frac{1}{k+1}} \left( \frac{2k}{k - (1 - \alpha)} \right)^\frac{k}{k+1}R^{\frac{k - (1 - \alpha)}{k+1}} + \frac{1}{R^2} E.\\
    \end{split}
\end{equation*}
    Choosing 
$$ 
R =
\left( \frac{2k}{k - (1 - \alpha)} \right)^{-\frac{k}{3k + 1 + \alpha}}
\left( \int f^{k+1} |xb| \right)^{- \frac 1 {3k + 1 + \alpha}} E^\frac{k+1}{3k + 1 + \alpha} \leq \bar x,
$$
we obtain
\begin{equation*}
    \begin{split}
        m &\leq 2 
        \left( \frac{2k}{k-(1 - \alpha)} \right)^\frac{2k}{3k + 1 + \alpha} \left(\int f^{k+1} |xb| \right)^\frac{2}{3k+1-\alpha}E^\frac{k - (1 - \alpha)}{3k + 1 + \alpha},
    \end{split}
\end{equation*}
which is estimate \eqref{Equa:cond_1_m_vs_E}.

\noindent {\it Case 2.} Similarly to the previous case, we compute
\begin{equation*}
    \begin{split}
        m &= \int f dx = \int_{|x| \leq M} f dx + \int_{|x| > M} f dx \\
        &\leq \int_{|x| \leq M} f |xb|^{\frac{1}{1+k}} \frac{1}{|xb|^{\frac{1}{1+k}}} dx + \frac{1}{M^2} E \\
     \big[ \text{if } \bar x \leq M \big] \quad   &\le \left(\int f^{k+1} |xb| dx \right)^{\frac{1}{k+1}} \left( \int_{|x| \leq \bar x} \left( \frac{1}{|x|}\right)^{\frac{1 - \alpha}{k}} dx + \int_{\bar x < |x| \leq M} \left( \frac{1}{|x|} \right)^{\frac{1 - \beta}{k}} \right)^{\frac{k}{k+1}} + \frac{1}{M^2}E \\
\big[ (\ref{Equa:bar_x_cond}) \big] \quad       &\le \left( \int f^{k+1} |xb| dx \right)^{\frac{1}{k+1}} \left( \frac{2k}{k -  (1 - \beta)} \right)^\frac{k}{k+1} M^{\frac{k-(1 - \beta)}{k+1}} + \frac{1}{M^2} E. 
    \end{split}
\end{equation*}

Choosing
$$
M := \left( \frac{2k}{k - (1 - \beta)} \right)^{-\frac{k}{3k+1+\beta}} \left( \int f^{k+1} |xb| \right)^{-\frac{1}{3k+1+\beta}} E^\frac{k+1}{3k+1+\beta},
$$
we conclude that
\begin{equation*}
    \begin{split}
        m &\leq 2 \left( \frac{2k}{k - (1 - \beta)} \right)^{\frac{2k}{3k + 1 + \beta}} \left( \int f^{k+1} |xb| dx \right)^{\frac{2}{3k + 1 + \beta}} E^{\frac{k - (1 - \beta)}{3k + 1 + \beta}},
    \end{split}
\end{equation*}
which is \eqref{Equa:m_esimta_beta}.
\end{proof}

\begin{theorem}
\label{prop:blowup}
     Assume \eqref{eq:con1} and  $k > 1 - \frac{1}{p}$. If $E(u_0)$ is sufficiently small and $b$ satisfying \eqref{eq:syst:b} and \eqref{Equa:bar_x_cond}, 
then the solution of \eqref{m:eq:pde} blows up in finite time.
\end{theorem}

\begin{proof}
The choice of the constant $\bar x$ covers both cases \eqref{eq:cond:1:c} and \eqref{eq:cond:2:c}: indeed 
\begin{equation*}
    \bar x^{3k+1+\beta} \left( \frac{2k}{k -(1 - \beta)} \right)^k \leq \left( - \int u^{k+1} xb \right)^{-1} E^{k+1} \le \bar x^{3k + 1 + \alpha} \left( \frac{2k}{k-(1 - \alpha)} \right)^k,
\end{equation*}
if and only if \eqref{Equa:bar_x_cond} holds.
Then by Proposition \ref{lemma:func1} and computing $\frac{dE}{dt}$ by \eqref{m:eq:pde}, one obtains 
\begin{equation*}
\begin{split}
\frac{dE}{dt} &\le 2 m - 2 \int x b u^{k+1} dx \\
&\le 2m - C \min \bigg\{ \frac{m^{\frac{3k + 1 + \alpha}{2}}}{E^{\frac{k - (1 - \alpha)}{2}}}, \frac{m^{\frac{3k + 1 + \beta}{2}}}{E^{\frac{k - (1 - \beta)}{2}}} \bigg\}.
\end{split}
\end{equation*}
The exponents 
$$
\frac{k - (1 - \alpha)}{2}, \frac{k - (1 - \beta)}{2} > 0
$$
by the choice of $\alpha, \beta$, which gives the blow-up in finite time because of the ODE
$$
\dot y = a - \frac{b}{y^\gamma}, \quad \gamma > 0
$$
has a solution converging to $0$ in finite time like
\begin{equation*}
    y(t) \simeq (T - t)^{\frac{1}{1+\gamma}}
\end{equation*}
if the initial data is $< (b/a)^{1/\gamma}$. This last condition reads as
\begin{equation*}
    E(t=0) \lesssim \min \big\{ m^{\frac{3k + \alpha - 2}{k - (1 - \alpha)}}, m^{\frac{3k + \beta - 2}{k - (1 - \beta)}} \big\}.
\end{equation*}

The final observation is that since $m(u)$ is constant, the fact that $E(u) \to 0$ forces $u$ to blow up in $L^\infty$ as the H\"older inequality implies.
\end{proof}

\subsection{Case \eqref{eq:con2}}\label{Ss:case_Sobolev}

Let $\alpha \in (0,1]$, $k > 1 + \alpha$ and consider a smooth odd bounded function $b$ with $b_x \in L^p$ such that for $x \geq 0$
\begin{equation}
\label{eq:syst:b:2}
    - \min\{ x^\alpha,  (\bar x)^\alpha \big\} \leq b(x) \leq - \min\{ x^\alpha, (2 \bar x)^\alpha \}, \quad \bar x \geq 1,
\end{equation}
and $(1 - \alpha) p < 1$.

\begin{proposition}
\label{lemma:func1:2}
For every measurable positive function $f : \R \to [0,\infty)$ the following holds. 
\begin{enumerate}
    \item If
    \begin{equation}
    \label{eq:cond:1:c:2}
        \left( \frac{2k}{k - (\alpha + 1)} \right)^{- \frac{k}{k + 1 - \alpha}} \left( \int f^{k+1} |xb| dx \right)^{- \frac{1}{3k + 1 - \alpha}} E^{\frac{k+1}{3k+1-\alpha}} \leq \bar x,
    \end{equation}
    then
    \begin{equation*}
        m \le 2 \left( \frac{2k}{k - (\alpha + 1)} \right)^{\frac{2k}{3k + 1 - \alpha}} \left( \int f^{k+1} |xb| dx \right)^{\frac{2}{3k + 1 - \alpha}} E^{\frac{k - (\alpha + 1)}{3k + 1 - \alpha}}.
    \end{equation*}
    \item If
    \begin{equation}
    \label{eq:cond:2:c:2}
        \left( \frac{2k}{k - 1} \right)^{-\frac{k}{3k + 1}}  \left( \int f^{k+1} |xb| dx \right)^{- \frac{1}{3k + 1}} E^{\frac{k + 1}{3k + 1}} \geq \bar x,
    \end{equation}
    then 
    \begin{equation*}
        m \le 2 \left( \frac{2k}{k - 1} \right)^{\frac{2k}{3k + 1}} \left( \int f^{k+1} |xb| dx \right)^{\frac{2}{3k + 1}} E^{\frac{k - 1}{3k + 1}}.
    \end{equation*}
\end{enumerate}
\end{proposition}

\begin{proof} {\it Case 1.} Define
\begin{equation}
    \label{Equa:def_R'}
R' :=  \left( \frac{2k}{k - (\alpha + 1)} \right)^{-\frac{k}{3k + 1 - \alpha}} \left( \int f^{k+1} |xb| dx \right)^{- \frac{1} {3k + 1-\alpha}} E^{\frac{k+1}{3k + 1 - \alpha}},
\end{equation}
and compute by H\"older inequality
\begin{equation*}
    \begin{split}
        m &= \int f dx = \int_{|x| \leq R'} f dx + \int_{|x| > R'} f dx \\
        &\le \left( \int f^{k+1} |x|^{\alpha + 1} dx \right)^{\frac{1}{k+1}} \left( \int_{|x| \leq R'} \frac{1}{|x|^{\frac{1+\alpha}{k}}} dx \right)^{\frac{k}{k+1}} + \frac{1}{{R'}^2} E \\
    \big[ (\ref{eq:cond:1:c:2}) \big] \quad   &= \left( \int f^{k+1} |xb| dx \right)^{\frac{1}{k+1}} \left( \frac{2k}{k-(\alpha+1)} \right)^{\frac{k}{k+1}} {(R')}^{\frac{k-(\alpha+1)}{k+1}} + \frac{1}{{(R')}^2} E \\
     \big[   (\ref{Equa:def_R'}) \big] \quad  &= 2 \left( \frac{2k}{k-(\alpha + 1)} \right)^{\frac{2k}{3k + 1 - \alpha}} \left( \int f^{k+1} |xb| dx \right)^{\frac{2}{3k + 1 - \alpha}} E^{\frac{k - (\alpha + 1)}{3k + 1 - \alpha}},
    \end{split}
\end{equation*}
    which is the first estimate.
    
\noindent{\it Case 2.}
Define
\begin{equation}
\label{Equa:def_M'}
M' := \left( \frac{2k}{k-1} \right)^{-\frac{k}{3k + 1}} \left( \int f^{k+1} |xb| dx \right)^{- \frac{1}{3k + 1}} E^{\frac{k+1}{3k + 1}},
\end{equation}
and compute
\begin{equation*}
    \begin{split}
        m &= \int f dx = \int_{|x| \leq M'} f dx + \int_{|x| > M'} f dx \\
        &\le \left( \int f^{k+1} |xb| dx \right)^{\frac{1}{k+1}} \left( \int_{|x| \leq M'} \frac{1}{|xb|^{\frac{1}{k}}} dx \right)^{\frac{k}{k + 1}} + \frac{1}{(M')^2} E \\
     \big[ (\ref{eq:syst:b:2}), (\ref{eq:cond:2:c:2}) \big] \quad  &\le \left( \int f^{k+1} |x|^{\alpha + 1} dx \right)^{\frac{1}{k+1}} \left( 2 \int_0^{\bar x} \frac{1}{|x|^{\frac{1 + \alpha}{k}}} dx + 2 \int_{\bar x}^{M'} \frac{1}{|x  (2 \bar x)^\alpha|^{\frac{1}{k}}} dx \right)^{\frac{k}{k+1}} + \frac{1}{(M')^2} E \\       
\big[ \bar x \geq 1 \big] \quad        &\le \left( \int f^{k+1} |xb| dx \right)^{\frac{1}{k+1}} \left( \frac{2k}{k - 1} \right)^{\frac{k}{k+1}}(M')^{\frac{k-1}{k+1}} + \frac{1}{(M')^2} E \\
  \big[ (\ref{Equa:def_M'} \big] \quad      &= 2 \left( \frac{2k}{k - 1} \right)^{\frac{2k}{3k + 1}} \left( \int f^{k+1} |xb| \right)^{\frac{2}{3k + 1}} E^{\frac{k - 1}{3k + 1}}.
    \end{split}
\end{equation*}
This is the second estimate in the statement.
\end{proof}

\begin{theorem} \label{prop:blowup:2} Assume \eqref{eq:con2} and  $k > 2-\frac{1}{p}$. If $E(u_0)\ll 1$ and $b$ satisfying \eqref{eq:syst:b} with
    \begin{equation}
    \label{Equa:cond_bar_x_2}
    \bar x \geq \left( \frac{k - 1}{k - (\alpha + 1)} \right)^{\frac{k}{\alpha}},
\end{equation}
then the solution of \eqref{m:eq:pde} blows up in finite time.
\end{theorem}

\begin{proof}
The choice of the constant $c$ covers both the cases \eqref{eq:cond:1:c:2} and \eqref{eq:cond:2:c:2}: indeed,
as in the proof of Theorem \ref{prop:blowup}, the condition \eqref{Equa:cond_bar_x_2} implies that
\begin{equation*}
\bar x^{3k + 1 - \alpha} \bigg( \frac{2k}{k - (1 + \alpha)} \bigg)^k \leq \bar x^{3k+1} \bigg( \frac{2k}{k-1} \bigg)^k.   
\end{equation*}
Then by Proposition \ref{lemma:func1:2} and \eqref{m:eq:pde}, one obtains
\begin{equation*}
    \frac{dE}{dt} \geq 2m - C \min \bigg\{ \frac{m^{\frac{3k + 1 -\alpha}{2}}}{E^{\frac{k - (\alpha + 1)}{2}}}, \frac{m^{\frac{3k+1}{2}}}{E^{\frac{k-1}{2}}} \bigg\},
    \end{equation*}
which leads to $E(t) \to 0$ if
\begin{equation*}
    E(t=0) \lesssim \min \big\{ m^{\frac{3k - 1 - \alpha}{k - (1 - \alpha)}}, m^{\frac{3k - 1}{k - 1}} \big\}.
\end{equation*}
As in the previous case, $E(t) \searrow$ in finite time implies $\|u(t)\|_\infty$ blows up.
\end{proof}

\begin{remark}
\label{Rem:blow_up_always}
Observe that if $u$ is bounded then
\begin{equation*}
\frac{\|u\|_1^3}{\|u\|_\infty^2} \lesssim E,
\end{equation*}
so that
\begin{equation*}
\frac{m^{\frac{2k - (1-\alpha)}{2}}}{E^{\frac{k-(1-\alpha)}{2}}} \gtrsim m^{1 - \alpha - \frac{k}{2}} \|u\|_\infty^{k-(1-\alpha)},
\end{equation*}
\begin{equation*}
\frac{m^{\frac{3k+1}{2}}}{E^{\frac{k-1}{2}}} \gtrsim \|u\|_\infty^{k-1} m^2.
\end{equation*}
Hence if $m \ll 1$, the blow up may not be possible.

This is also easily verified directly by the estimate (obtained as in \eqref{m:eq:f3}) 
\begin{equation*}
\begin{split}
\frac{d}{dt} \|u\|_2^2 &\leq - \|u_x\|_2^2 + C \|b\|_{p,\infty} \|u_x\|_2 \|u\|_{(k+1)\tilde p'}^{k+1} \\
&\leq - \|u_x\|_2^2 + C \|u\|_1^{\frac{1}{\tilde p'}} \|u_x\|_2 \|u\|_\infty^{k+1 - \frac{1}{\tilde p'}}.
\end{split}
\end{equation*}
Thus $\|u\|_2^2$ is bounded if
\begin{equation*}
\|u\|_1^{\frac{1}{p'}} \|u\|_\infty^{k+1 - \frac{1}{\tilde p}} \lesssim \frac{\|u\|_\infty^{\frac{3}{2}}}{\|u\|_1^\frac{1}{2}} \lesssim \|u_x\|_2, \quad \|u\|_\infty^{k-\frac{1}{2} - \frac{1}{p'}} \lesssim \frac{1}{\|u\|_1^{\frac{1}{2} + \frac{1}{p'}}}.
\end{equation*}
where we have used Gagliando-Niremberg inequality. Hence using again the bound \eqref{Equa:guidolin} we have that this bound can be prolonged up to $+\infty$ if $\|u\|_1 \ll 1$.

A completely similar estimate can be done for the case $b_x \in L^\infty L^{p,\infty}$, $p \in (1,\infty]$.
\end{remark}

\section{Long behavior of solutions}
\label{s:long:behav}

In this last section, we discuss the long behavior of the solutions when there are uniform bounds on the drift $b$:
\begin{enumerate}
    \item $b \in L^\infty \big( (0,\infty), L^{p,\infty}(\R) \big)$ for the case (\ref{eq:con1}),
\item $b_x \in L^\infty \big( (0,\infty), L^{p,\infty}(\R) \big)$,
$b \in L^\infty_{\loc}((0,\infty) \times \R \big)$ for the case (\ref{eq:con2}).
\end{enumerate}

We prove the following theorem.

\begin{theorem}
    \label{Theo:long_time_1}
    The following holds.
    \begin{enumerate}
        \item \label{Point_1:long_time_1} If $k < 1 - \frac{1}{p}$, then there are drifts $b = b(x) \in L^{p,\infty}$ admitting a stationary solution; similarly, if $k < 2 - \frac{1}{p}$, there are drift $b_x(x) \in L^{p,\infty}_x$ admitting a stationary solution.


        \item \label{Point_3:long_time_1} If $k \geq 1 - \frac{1}{p}$, $b \in L^\infty_t L^{p,\infty}_x$ or $k \geq 2 - \frac{1}{p}$, $b_x \in L^\infty_t L^{p,\infty}_x$, every uniformly bounded solution $u(t)$ decays to $0$ as $t^{-\frac{1}{2}}$ in the $L^\infty$-norm.
    \end{enumerate}
\end{theorem}

\begin{proof} 
{\it Point \eqref{Point_1:long_time_1}.} Define the function
\begin{equation*}
    u(x) = 
    \frac{1}{(1 + x^2)^{\frac{1-1/p}{2k}}},
\end{equation*}
which is a stationary solution to \eqref{m:eq:pde} in $L^1$ if $k < 1 - \frac{1}{p}$ and the drift $b$ is given by
\begin{equation*}
    b(x) = - \frac{1 - \frac{1}{p}}{k} \frac{x}{(1 + x^2)^{\frac{1 + 1/p}{2}}}.
\end{equation*}
Being $b(x) \sim |x|^{-1/p}$ for $|x| \gg 1$, we have that $b \in L^{p,\infty}(\R)$.

For the case $b_x \in L^{p,\infty}$, one can similarly show that
\begin{equation*}
    u(x) = \frac{1}{(1 + x^2)^{\frac{2 - \frac{1}{p}}{2k}}},
\end{equation*}
which gives the drift
\begin{equation*}
    b(x) = - \frac{2 - \frac{1}{p}}{k} \frac{x}{(1 + x^2)^{\frac{1}{2p}}}.
\end{equation*}

\medskip

\noindent {\it Point \eqref{Point_3:long_time_1}.} We have only to consider the supercritical case, being $k = 1 - \frac{1}{p}$ studied in Section \ref{Ss:critical_case_integr}.

For the case $b \in L^\infty_t L_x^{p,\infty}$, we observe that
\begin{equation*}
b u^{1 + k} = (b u^{k- 1 + \frac{1}{p}}) u^{1 - \frac{1}{p}} = \hat b u^{1 - \frac{1}{p}},
\end{equation*}
with
\begin{equation*}
\|\hat b\|_{p,\infty} \leq \|u\|_\infty^{k - 1 + \frac{1}{p}} \|b\|_{p,\infty}.
\end{equation*}
Hence the analysis of the critical case can be applied here, deducing that $\|u(t)\|_\infty \leq \frac{C}{\sqrt{t}}$.

For the case $b_x \in L^{p,\infty}$, we follow the computations of \eqref{m:eq:f2}:
\begin{equation*}
\begin{split}
\bigg| \int v_{a,y} \hb v_a^{k+1} \bigg| &\leq \frac{1}{k+2} \bigg| \int \hb_y v_a^{k+2} \bigg| \\
&\leq C \|\hb_y\|_{p,\infty} \|v_a^{k+2}\|_{p',1} \\
&\leq C \|\hb_y\|_{p,\infty} \|v_a\|_\infty^{k - 2 + \frac{1}{p}} \|v_a^{2 - \frac{1}{p} + 2}\|_{p',1} \\
\big[ \|v(\tau)\|_\infty \lesssim (T e^{-\tau})^{\frac{1}{2}} \big] \quad &\leq C (T e^{-\tau})^{\frac{2 - \frac{1}{p} - k}{2}} (T e^{-\tau})^{\frac{k - 2 + \frac{1}{p}}{2}} \|v_a\|_{(2 - \frac{1}{p})p',2}^{2+2 - \frac{1}{p}} \\
\big[ \text{as in (\ref{m:eq:f2})} \big] \quad &\leq \frac{1}{2} \|v_{a,y}\|_2^2 + \frac{1}{8} \|v_a\|_2^{2}.
\end{split}
\end{equation*}
Hence one can repeat the same analysis as in the critical case $k = 2 - \frac{1}{p}$, replacing $\|b_x\|_{p,\infty}$ with $\|b_x\|_\infty \|u\|_\infty^{k - 2 + \frac{1}{p}}$. In particular one obtains the decay $\|u(t)\|_\infty \lesssim t^{-\frac{1}{2}}$. 
\end{proof}

\begin{remark}
    \label{Rem:bounde_L_p}
    By slightly changing the exponents in the subcritical case, it is possible to show that actually the vector field can be taken in $L^p$ (or $b_x \in L^p$) in Point \eqref{Point_1:long_time_1} of the above theorem.
\end{remark}

\printbibliography

\end{document}